\documentclass[11pt,british,refpage,intoc,bibliography=totoc,index=totoc,BCOR=7.5mm,captions=tableheading]{extarticle}

\usepackage[T1]{fontenc}
\usepackage[latin9]{inputenc}
\usepackage[a4paper]{geometry}
\usepackage{color}
\usepackage{babel}
\usepackage{amsmath}
\usepackage{amsthm}
\usepackage{amssymb}
\usepackage{stmaryrd}
\usepackage[numbers]{natbib}
\usepackage[unicode=true,pdfusetitle,
 bookmarks=true,bookmarksnumbered=true,bookmarksopen=false,
 breaklinks=false,pdfborder={0 0 0},pdfborderstyle={},backref=false,colorlinks=true]
 {hyperref}
\hypersetup{
 linkcolor=black, citecolor=blue, urlcolor=black, filecolor=black, pdfpagelayout=OneColumn, pdfnewwindow=true, pdfstartview=XYZ, plainpages=false}
\usepackage{breakurl}

\makeatletter
\numberwithin{equation}{section}
\numberwithin{figure}{section}
\theoremstyle{plain}
\newtheorem{thm}{\protect\theoremname}[section]
  \theoremstyle{definition}
  \newtheorem{defn}[thm]{\protect\definitionname}
  \theoremstyle{remark}
  \newtheorem{rem}[thm]{\protect\remarkname}
  \theoremstyle{plain}
  \newtheorem{lem}[thm]{\protect\lemmaname}
  \theoremstyle{plain}
  \newtheorem{cor}[thm]{\protect\corollaryname}
  \theoremstyle{plain}
  \newtheorem{prop}[thm]{\protect\propositionname}
  \theoremstyle{definition}
  \newtheorem*{example*}{\protect\examplename}

\@ifundefined{date}{}{\date{}}
\usepackage{caption}
\usepackage[nottoc]{tocbibind}

\makeatother

  \providecommand{\corollaryname}{Corollary}
  \providecommand{\definitionname}{Definition}
  \providecommand{\examplename}{Example}
  \providecommand{\lemmaname}{Lemma}
  \providecommand{\propositionname}{Proposition}
  \providecommand{\remarkname}{Remark}
\providecommand{\theoremname}{Theorem}

\begin{document}

\title{Backward problems for stochastic differential equations on the Sierpinski
gasket}

\author{Xuan~Liu\thanks{Mathematical Institute, University of Oxford, Oxford, OX2 6GG, United
Kingdom. Email: \protect\href{mailto:xuan.liu@maths.ox.ac.uk}{xuan.liu@maths.ox.ac.uk}}\hspace{0.7em}and Zhongmin~Qian\thanks{Research supported partly by an ERC grant(Grant Agreement No. 291244
ESig). Mathematical Institute, University of Oxford, Oxford, OX2 6GG,
United Kingdom. Email: \protect\href{mailto:zhongmin.qian@maths.ox.ac.uk}{zhongmin.qian@maths.ox.ac.uk}}}
\maketitle
\begin{abstract}
In this paper, we study the non-linear backward problems (with deterministic
or stochastic durations) of stochastic differential equations on the
Sierpinski gasket. We prove the existence and uniqueness of solutions
of backward stochastic differential equations driven by Brownian martingale
(defined in Section \ref{sec:-1}) on the Sierpinski gasket constructed
by S.~Goldstein and S.~Kusuoka. The exponential integrability of
quadratic processes for martingale additive functionals is obtained,
and as an application, a Feynman-Kac representation formula for weak
solutions of semi-linear parabolic PDEs on the gasket is also established.
\end{abstract}
\bigskip{}

\emph{\small{}}%
\begin{minipage}[t]{0.7\paperwidth}%
\textbf{\small{}Keywords}{\small{}\enskip{}Sierpinski gasket, backward
stochastic differential equations, semi-linear parabolic equations\medskip{}
}{\small \par}

\textbf{\small{}Mathematics Subject Classification}{\small{} }\textbf{\small{}(2000)}{\small{}\enskip{}28A80,
60H10, 60H30}{\small \par}%
\end{minipage}{\small \par}

\section{\label{sec:}Introduction}

Interests on diffusions on fractals initially came from mathematical
physics such as percolation clusters near the percolation thresholds.
(See \citep{RT83} and references therein.) Diffusions on fractals
are also of significance from mathematical point of views:\smallskip{}

\ (i) Calculus can be established on a manifold with a differential
structure, while there are interesting spaces appearing in applications
which possess no suitable smooth structures. It is important and tempting
to consider calculus on fractals as archetypical examples of singular
spaces; 

(ii) By utilizing the general theory of Dirichlet forms, linear analysis
may be well studied via the corresponding Markov semigroups in a quite
general setting.

\smallskip{}

Dynamic systems in physics (e.g. fluid dynamics) and mathematical
finance (e.g. optimal stochastic controls) are usually described by
non-linear partial differential equations. As suggested by the situation
on Euclidean spaces, knowledge on non-linear PDEs on the Sierpinski
gasket can be obtained by studying the corresponding stochastic dynamic
systems, which will be the subject of this paper as a part of an investigation
of non-linear analysis on fractals.

Brownian motion on the Sierpinski gasket was first constructed by
S.~Goldstein \citep{Gold87} and S.~Kusuoka \citep{Ku87} as the
limit of sequences of random walks. A similar construction was used
by M.~Barlow and E.~Perkins in \citep{BP88}, where heat kernel
estimates were obtained. Brownian motion on the gasket is a diffusion
process symmetric with respect to the Hausdorff measure, and therefore,
gives rise to a Dirichlet form. The standard Laplacian on the Sierpinski
gasket is defined to be the associated self-adjoint operator. J.~Kigami
\citep{Ki89} gave an analytic construction of the Dirichlet form
using products of stochastic matrices. There have been many works
on the study of Brownian motion and the Laplacian on the Sierpinski
gasket (see, for example, \citep{BP88,FS92,KL93,BK97,Ki98}, and etc.).
Several definitions of gradients of functions with finite energies
have been introduced and studied with applications to the study of
non-linear PDEs and differential forms on fractals (cf. \citep{Ki93,Str00,Tep00,CS03,Hin10,HRT13,HT15,BK16}
and references in these works). See Remark \ref{rem:-2}-(ii) for
more details and comments on the connection between the definition
of gradients introduced in Definition \ref{def:-2-1} below and those
in the aforementioned literature.

In this paper, we develop a theory of backward stochastic differential
equations (BSDEs hereafter for short) on the Sierpinski gasket. As
an application, we derive a representation formula for solutions of
semi-linear parabolic PDEs on the gasket, which will formulated later
in Section \ref{sec:-6}. In contrast to the BSDE theory on Euclidean
spaces, BSDEs on the gasket have two mutually singular drift terms,
which is due to the singularity between the Hausdorff measure as the
volume measure and the Kusuoka measure as the energy measure (see
Section \ref{sec:-1} for more details). Singularity of different
drift terms introduces significant difficulties in the study of BSDEs
on the gasket, one of which is the exponential integrability of a
quadratic process (cf. Theorem \ref{thm:-1-1} below). Our approach
to the exponential integrability is based on moment estimates for
the quadratic processes by expressing the moments as iterated integrals
and by using the heat kernel estimate. The study of BSDEs on the gasket
here is a specific case and a continuation of the papers \citep{KR13,KR16}
investigating BSDEs associated with Dirichlet forms and their application
to partial differential equations. (See Section \ref{sec:-6} for
more related works and connections between BSDEs studied in the current
paper and those in literature.) Though results in this paper are stated
and proved specifically for $2$-dimensional Sierpinski gasket, we
however believe that our results also hold for higher-dimensional
Sierpinski gaskets, and the proofs given in this paper should be easily
adapted to higher-dimensional cases.

The paper is organized as follows. In Section \ref{sec:-1}, we introduce
notations used in the paper and recall several results that will be
needed in following sections. The main results are formulated in Section
\ref{sec:-6}. Section \ref{sec:-2} is devoted to several results
on Brownian motion which will be needed in later sections. In particular,
a representation theorem for square-integrable martingales is given
as an immediate consequence of the results in \citep{Ku89}, \citep{Hin08}
and \citep{QY12}. A result on the exponential integrability for quadratic
processes is also given in this section, which is of interests in
itself. In Section \ref{sec:-3}, we prove the existence and uniqueness
of solutions to BSDEs with deterministic or stochastic durations.
In the last section, Section \ref{sec:-4}, we give the proof of a
Feynman-Kac representation of solutions of semi-linear parabolic PDEs,
and as a result, derive the uniqueness of weak solutions.

\section{\label{sec:-1}Notations and several related results}

In this section, we introduce several notations and notions which
will be in force throughout this paper.

\subsection{\label{subsec:}Diffusions and Dirichlet forms on the gasket}

Let $\mathrm{V}_{0}=\{p_{1},p_{2},p_{3}\}\subseteq\mathbb{R}^{2}$
with $p_{1}=(0,0),\;p_{2}=(1,0),\;p_{3}=(\tfrac{1}{2},\tfrac{\sqrt{3}}{2})$,
and $F_{i}:\mathbb{R}^{2}\to\mathbb{R}^{2},\;i=1,2,3$ be the contraction
mappings given by $F_{i}(x)=\tfrac{1}{2}(x+p_{i}),\;x\in\mathbb{R}^{2},\;i=1,2,3.$
Let $\mathrm{W}_{\ast}=\big\{\omega=\omega_{1}\omega_{2}\omega_{3}\dots\,:\omega_{i}\in\{1,2,3\},\;i\in\mathbb{N}_{+}\big\}$.
For $\omega=\omega_{1}\omega_{2}\omega_{3}\dots\,\in\mathrm{W}_{\ast}$
and $m\in\mathbb{N}_{+}$, we denote $[\omega]_{m}=\omega_{1}\cdots\omega_{m}$
and $F_{[\omega]_{m}}=F_{\omega_{1}}\circ F_{\omega_{2}}\circ\cdots\circ F_{\omega_{m}}$.
Let $\mathrm{V}_{m}=\bigcup_{\omega\in\mathrm{W}_{\ast}}F_{[\omega]_{m}}(\mathrm{V}_{0}),\;m\in\mathbb{N}_{+}$,
and $\mathrm{V}_{\ast}=\bigcup_{m=0}^{\infty}\mathrm{V}_{m}$. The
($2$-dimensional) \emph{Sierpinski gasket} is defined to be the closure
$\mathbb{S}=\mathrm{cl}(\mathrm{V}_{\ast})$. 

The Dirichlet form on $\mathbb{S}$ can be introduced via a finite
difference scheme described below. For any functions $u,v$ on $\mathrm{V}_{\ast}$,
define
\[
\mathcal{E}^{(m)}(u,v)=\sum_{x,y\in\mathrm{V}_{m},|x-y|=2^{-m}}\;\frac{1}{2}\,\Big(\frac{5}{3}\Big)^{m}\,[u(x)-u(y)][v(x)-v(y)],\;\;m\in\mathbb{N}.
\]
For each function $u$ on $\mathrm{V}_{\ast}$, $\mathcal{E}^{(m)}(u,u)$
is non-decreasing in $m$, and $\mathcal{E}(u,u)=\lim_{m\to\infty}\mathcal{E}^{(m)}(u,u)$
exists (possibly infinite). Let $\mathcal{F}(\mathbb{S})=\{u:\mathcal{E}(u,u)<\infty\}$.
Every $u\in\mathcal{F}(\mathbb{S})$ is continuous on $\mathrm{V}_{\ast}$
and therefore can be extended to a continuous function on $\mathbb{S}$
(\citep[Section 2.2 and Section 3.1]{Ki01}). In other words, $\mathcal{F}(\mathbb{S})\subseteq C(\mathbb{S})$,
which in fact follows easily from the following
\begin{equation}
\mathop{\mathrm{osc}}_{\mathbb{S}}(u)\le C_{\ast}\sqrt{\mathcal{E}(u,u)},\quad u\in\mathcal{F}(\mathbb{S}),\label{eq:-40}
\end{equation}
where $C_{\ast}>0$ is a universal constant. (See \citep[Lemma 2.3.9 and Theorem 3.3.4]{Ki01}.)
Since $\mathcal{F}(\mathbb{S})\subseteq C(\mathbb{S})$, the empty
set $\emptyset$ is the only subset of $\mathbb{S}$ with zero capacity.
As a consequence of the definition of $\mathcal{E}^{(m)}$, it is
seen that $\mathcal{E}^{(m+1)}(u,v)=\sum_{i=1,2,3}\frac{5}{3}\mathcal{E}^{(m)}(u\circ F_{i},v\circ F_{i}),\;m\in\mathbb{N}$,
which implies the following self-similar property of $\mathcal{E}$:
\begin{equation}
\mathcal{E}(u,v)=\sum_{i=1,2,3}\frac{5}{3}\,\mathcal{E}(u\circ F_{i},v\circ F_{i}),\;\;u,v\in\mathcal{F}(\mathbb{S}).\label{eq:-30}
\end{equation}

For any function $u$ on $\mathrm{V}_{0}$, there exists a unique
$h\in\mathcal{F}(\mathbb{S})$ such that $h|_{\mathrm{V}_{0}}=u$
and 
\[
\mathcal{E}(h,h)=\min\{\mathcal{E}(v,v):v\in\mathcal{F}(\mathbb{S})\;\text{and}\ v|_{\mathrm{V}_{0}}=u\}.
\]
 (See \citep[Corollary 3.2.15]{Ki01}.) The above function $h$ is
called the \emph{harmonic function }in $\mathbb{S}$ with boundary
value $u$, and denoted by $h=Hu$. For the harmonic function $Hu$
with boundary value $u$, we have
\begin{equation}
\mathcal{E}(Hu,Hu)=\mathcal{E}^{(0)}(u,u)=\frac{3}{2}\,u^{t}\,\mathbf{P}u,\label{eq:-32}
\end{equation}
where
\begin{equation}
\mathbf{P}=\frac{1}{3}\left[\begin{array}{ccc}
2 & -1 & -1\\
-1 & 2 & -1\\
-1 & -1 & 2
\end{array}\right].\label{eq:-35}
\end{equation}
The values of $Hu$ on $\mathrm{V}_{\ast}$ is given by
\[
(Hu)\circ F_{[\omega]_{m}}=\mathbf{A}_{[\omega]_{m}}u\;\;\text{for all}\ \omega\in\mathrm{W}_{\ast},m\in\mathbb{N},
\]
where 
\begin{equation}
\mathbf{A}_{[\omega]_{m}}=\mathbf{A}_{\omega_{m}}\mathbf{A}_{\omega_{m-1}}\cdots\mathbf{A}_{\omega_{1}}\label{eq:-36}
\end{equation}
with $\mathbf{A}_{i},\,i=1,2,3$ being the linear operators having
matrix representations
\begin{equation}
\mathbf{A}_{1}=\frac{1}{5}\left[\begin{array}{ccc}
5 & 0 & 0\\
2 & 2 & 1\\
2 & 1 & 2
\end{array}\right],\;\mathbf{A}_{2}=\frac{1}{5}\left[\begin{array}{ccc}
2 & 2 & 1\\
0 & 5 & 0\\
1 & 2 & 2
\end{array}\right],\;\mathbf{A}_{3}=\frac{1}{5}\left[\begin{array}{ccc}
2 & 1 & 2\\
1 & 2 & 2\\
0 & 0 & 5
\end{array}\right].\label{eq:-33}
\end{equation}
Notice that, in contrast to the definition of $F_{[\omega]}$, the
order of matrices $\mathbf{A}_{\omega_{i}}$ on the right hand side
of (\ref{eq:-36}) is reversed.\medskip{}

Let $\mu$ be the (normalized) \emph{Hausdorff measure} on $\mathbb{S}$;
that is, $\mu$ is the unique Borel probability measure on $\mathbb{S}$
such that $\mu(F_{[\omega]_{m}}(\mathbb{S}))=3^{-m}$ for all $\omega\in\mathrm{W}_{\ast},\,m\in\mathbb{N}$.
Then $(\mathcal{E},\mathcal{F}(\mathbb{S}))$ is a regular Dirichlet
form on $L^{2}(\mathbb{S};\mu)$, called the \emph{standard Dirichlet
form} on $\mathbb{S}$, of which the associated non-positive self-adjoint
operator will be denoted by $\mathcal{L}$. According to the theory
of Dirichlet forms and Markov processes (see \citep[Chapter 7]{FOT10}),
there exists a standard Hunt process $(\Omega,\mathcal{F},\{X_{t}\},\{\theta_{t}\},\{\mathbb{P}_{x}\}_{x\in\mathbb{S}\cup\{\Delta\}})$
with state space $\mathbb{S}$, where $\Delta$ is an isolated point
adjoined to $\mathbb{S}$, and $\theta_{t}:\Omega\to\Omega,\;t\ge0$
are the shift operators. The Hunt process $\{X_{t}\}$ is a diffusion
on $\mathbb{S}$, called (the reflected) \emph{Brownian motion }on
$\mathbb{S}$, of which the associated Markov semigroup will be denoted
by $\{P_{t}\}$. Let $\mathcal{P}(\mathbb{S})$ be the family of all
Borel probability measures on $\mathbb{S}$. For each $\lambda\in\mathcal{P}(\mathbb{S})$,
the probability measure $\mathbb{P}_{\lambda}$ on $\Omega$ is defined
by 
\[
\mathbb{P}_{\lambda}(E)=\int_{\mathbb{S}}\mathbb{P}_{x}(E)\lambda(dx),\;E\in\mathcal{F}.
\]
The expectation with respect to $\mathbb{P}_{\lambda}$ will be denoted
by $\mathbb{E}_{\lambda}$, and we denote by $\{\mathcal{F}_{t}\}$
the \emph{minimal admissible filtration }determined by $\{X_{t}\}$;
that is, $\mathcal{F}_{t}=\bigcap_{\lambda\in\mathcal{P}(\mathbb{S})}\mathcal{F}_{t}^{\lambda},\;t\ge0$,
where $\mathcal{F}_{t}^{\lambda}$ is the $\mathbb{P}_{\lambda}$-completion
of $\sigma\left(X_{r}:r\le t\right)$ in $\mathcal{F}$. An $\{\mathcal{F}_{t}\}$-adapted
process $A_{t}$ is called an \emph{additive functional} if $A_{t+s}=A_{s}+A_{t}\circ\theta_{s}$
for all $t,s\ge0$.\medskip{}

Let $\mathcal{F}(\mathbb{S}\backslash\mathrm{V}_{0})=\{u\in\mathcal{F}(\mathbb{S}):u|_{\mathrm{V}_{0}}=0\}$.
The restriction of $\mathcal{E}$ on $\mathcal{F}(\mathbb{S}\backslash\mathrm{V}_{0})$
is also a Dirichlet form with $\mathcal{F}(\mathbb{S}\backslash\mathrm{V}_{0})$
as its Dirichlet space. Dirichlet spaces $\mathcal{F}(\mathbb{S})$
and $\mathcal{F}(\mathbb{S}\backslash\mathrm{V}_{0})$ correspond
to the Neumann boundary conditions and the Dirichlet boundary conditions
respectively. (See \citep[Theorem 3.7.9]{Ki01}.) Let $\sigma_{\mathrm{V}_{0}}$
be the hitting time $\sigma_{\mathrm{V}_{0}}=\inf\left\{ t>0:X_{t}\in\mathrm{V}_{0}\right\} $.
We define the \emph{killed Brownian motion} $\{X_{t}^{0}\}$ by killing
$\{X_{t}\}$ on hitting\emph{ }$\mathrm{V}_{0}$; that is, $X_{t}^{0}=X_{t}$
if $t<\sigma_{\mathrm{V}_{0}}$, and $X_{t}^{0}=\Delta$ if $t\ge\sigma_{\mathrm{V}_{0}}$.
Then $\{X_{t}^{0}\}$ is a $\mu$-symmetric Hunt process on $\mathbb{S}\backslash\mathrm{V}_{0}$
with $\big(\mathcal{E},\mathcal{F}(\mathbb{S}\backslash\mathrm{V}_{0})\big)$
as its associated Dirichlet form, and the associated semigroup, denoted
by $\{P_{t}^{0}\}$, is given by $P_{t}^{0}(x,E)=\mathbb{P}_{x}\big(X_{t}\in E,\;t<\sigma_{\mathrm{V}_{0}}\big)$
for all $x\in\mathbb{S}\backslash\mathrm{V}_{0},\;E\in\mathcal{B}(\mathbb{S}\backslash\mathrm{V}_{0})$.
It can be easily shown that if $A_{t}$ is an additive functional,
then $A_{t}^{\sigma_{\mathrm{V}_{0}}}=A_{t\wedge\sigma_{\mathrm{V}_{0}}}$
is an additive functional with the shift operators $\theta_{t}$ replaced
by $\theta_{t\wedge\sigma_{\mathrm{V}_{0}}}$.

\subsection{\label{subsec:-3}Kusuoka measure and Brownian martingale}

Let $\mathbf{P}$ and $\mathbf{A}_{i},\,i=1,2,3$ be the matrices
in (\ref{eq:-35}) and (\ref{eq:-33}) respectively, and let $\mathbf{Y}_{i}=\mathbf{P}\mathbf{A}_{i}\mathbf{P},\;i=1,2,3$.
\begin{defn}
The \emph{Kusuoka measure} $\nu$ is defined to be the unique probability
measure $\nu$ on $\mathbb{S}$ such that
\[
\nu\big(F_{[\omega]_{m}}(\mathbb{S})\big)=\frac{1}{2}\,\left(\frac{5}{3}\right)^{m}\;\mathrm{tr}\big(\mathbf{Y}_{[\omega]_{m}}^{t}\mathbf{Y}_{[\omega]_{m}}\big)\;\;\text{for all}\;\omega\in\mathrm{W}_{\ast},m\in\mathbb{N},
\]
where, similar to (\ref{eq:-36}), $\mathbf{Y}_{[\omega]_{m}}=\mathbf{Y}_{\omega_{m}}\mathbf{Y}_{\omega_{m-1}}\cdots\mathbf{Y}_{\omega_{1}}$.
\end{defn}
\begin{rem}
\label{rem:-6}The Kusuoka measure $\nu$ is singular to the Hausdorff
measure $\mu$. (See \citep[Corollary (2.15) and Example 1, p. 678]{Ku89}.)
\end{rem}
\begin{defn}
\label{def:-5}For each $u\in\mathcal{F}(\mathbb{S})$, the \emph{energy
measure} $\nu_{\langle u\rangle}$ of $u$ is defined to be the unique
Borel measure on $\mathbb{S}$ such that
\[
\int_{\mathbb{S}}\phi\,d\nu_{\langle u\rangle}=\mathcal{E}(\phi u,u)-2\mathcal{E}(\phi,u^{2})\;\;\text{for all}\;\phi\in\mathcal{F}(\mathbb{S}).
\]
For any $u,v\in\mathcal{F}(\mathbb{S})$, the \emph{mutual energy
measure $\nu_{\langle u,v\rangle}$} of $u$ and $v$ is defined by
polarization $\nu_{\langle u,v\rangle}=\frac{1}{4}\big(\nu_{\langle u+v\rangle}-\nu_{\langle u-v\rangle}\big)$.
\end{defn}
\begin{rem}
\label{rem:-7}The measure $\nu$ is an energy dominant; that is,
$\nu_{\langle u\rangle}\ll\nu$ for all $u\in\mathcal{F}(\mathbb{S})$.
(See \citep[Lemma (5.1)]{Ku89}.)
\end{rem}
\begin{defn}
For any positive additive functional $A_{t}$, the \emph{Revuz measure}
$\nu_{A}$ of $A$ is defined to be the unique Radon measure on $\mathbb{S}$
such that 
\[
\int_{\mathbb{S}}\phi\,d\nu_{A}=\lim_{t\to0}\frac{1}{t}\,\mathbb{E}_{\mu}\Big(\int_{0}^{t}\phi(X_{r})\,dA_{r}\Big)\;\;\text{for all}\ \phi\in\mathcal{B}_{+}(\mathbb{S}).
\]
\end{defn}
According to stochastic analysis for Brownian motions on $\mathbb{R}^{d}$,
any square-integrable martingale adapted to the Brownian filtration
can be written as an Itô integral against Brownian motion, which is
the crucial ingredient in solving SDEs backward. On fractal, representations
for martingale additive functionals were first proved by S.~Kusuoka
\citep{Ku89}, and generalized to a large class of self-similar sets
including nested fractals by M.~Hino \citep{Hin08}. The following
representation theorem is the specific case of \citep[Theorem (5.4)]{Ku89}
applied to the Sierpinski gasket.
\begin{thm}
\label{thm:-}There exists a martingale additive functional $W_{t}$
satisfying the following:

(i) $W_{t}$ has $\nu$ as its energy measure; that is, $\nu_{\langle W\rangle}=\nu$;

(ii) for any $u\in\mathcal{F}(\mathbb{S})$, there exists a unique
$\zeta\in L^{2}(\mathbb{S};\nu)$ such that 
\begin{equation}
M_{t}^{[u]}=\int_{0}^{t}\zeta(X_{r})dW_{r}\;\;\text{for all}\ t\ge0,\label{eq:-37}
\end{equation}
where $M^{[u]}$ is the martingale part of $u(X_{t})-u(X_{0})$.
\end{thm}
\begin{rem}
\label{rem:-9}Notice that the martingale additive functional $W_{t}$
in Theorem \ref{thm:-} is not unique. In fact, let $g$ be any Borel
measurable function on $\mathbb{S}$ such that $|g|=1$. Then $\int_{0}^{t}g(X_{r})\,dW_{r}$
is also a martingale additive functional satisfying the properties
in Theorem \ref{thm:-}. However, we may have a canonical choice if
the sign of the integrand $\zeta\in L^{2}(\mathbb{S};\nu)$ in (\ref{eq:-37})
is specified. More precisely, we make the following definition.
\end{rem}
\begin{defn}
\label{def:-3}The \emph{Brownian martingale} is defined to be the
unique martingale additive functional $W_{t}$ such that $W_{t}$
satisfies the properties in Theorem \ref{thm:-}, and that, if $h$
is the harmonic function with boundary value $1_{\{p_{1}\}}$ and
$M_{t}^{[h]}=\int_{0}^{t}\zeta(X_{r})\,dW_{r}$, then $\zeta<0\;\;\nu$-a.e.
\end{defn}
\begin{rem}
\label{rem:-8}The sign of $\zeta$ in Definition \ref{def:-3} is
chosen as a convention.
\end{rem}
We shall always denote by $W_{t}$ the Brownian martingale in the
rest of this paper.

\section{\label{sec:-6}Formulation of main results}

Linear BSDEs were first introduced by J.~Bismut to establish the
Pontryagin maximum principle in stochastic control theory (see \citep{Bis78,Yong99}
and etc\emph{.}), while the theory of non-linear BSDEs was developed
in E.~Pardoux and S.~Peng \citep{PP90}. The celebrated Feynman-Kac
formula was also generalized to non-linear cases in \citep{Pen91}
using BSDEs. There have been a large amount of works on the theory
of BSDEs. For example, BSDEs associated with (non-symmetric) second-order
elliptic operators of divergence forms on Euclidean spaces were studied
in \citep{Sto03,BPS05} and applied to study semi-linear parabolic
PDEs involving divergences of measurable vector fields, where an Itô-Fukushima
decomposition for the diffusion process associated to the elliptic
operator was derived in terms of forward-backward martingales, and
a representation formula for solutions of parabolic PDEs was obtained.
BSDEs and semi-linear parabolic equations on Hilbert spaces were investigated
in \citep{Zhu15} using methods from functional analysis and generalized
Dirichlet forms. A martingale representation with countably many representing
martingales for the infinite dimensional case was also proved in \citep{Zhu15}
in order to solve BSDEs on Hilbert spaces. In \citep{KR13,KR16},
existence and uniqueness of solutions to a class of BSDEs associated
with (not necessarily local) regular (or quasi-regular) Dirichlet
forms were established, together with a probabilistic representation
of solutions to semi-linear elliptic equations perturbed by smooth
measures. It was also shown in \citep{KR13} that the probabilistic
solutions yielded by BSDEs coincide with the notion of weak solutions
(called \emph{solutions in the sense of duality} in \citep{KR13})
under a transience assumption on the Dirichlet form.

In the current paper, we consider BSDEs, with deterministic or stochastic
durations, driven by Brownian martingale $W_{t}$. We shall give a
Feynman-Kac representation for solutions of semi-linear parabolic
equations on $\mathbb{S}$, of which the meaning will be formulated
later. On the one hand, the BSDEs considered here can be regarded
as a specific case of those studied in \citep{KR13,KR16} for quite
general (quasi-regular) Dirichlet forms. On the other hand, solutions
to BSDEs in our case can be formulated in a more specific way (e.g.
the martingale parts are given as stochastic integrals, and the exponential
integrability assumption on quadratic processes as drifts can be verified),
which is due to the specific setting of the gasket. Regarding probabilistic
interpretations of parabolic equations in terms of BSDEs, the representation
in Theorem \ref{thm:-3-1} is an analogue of the representations established
in \citep{Pen91,Sto03,BPS05}.

To simplify notations, we shall adopt the convention
\[
u(t)=u(t,\cdot)\;\;\text{and}\;\;g(t,u(t))=g(t,u(t,\cdot),\cdot)
\]
for any function $u:[0,\infty)\times\mathbb{S}\to\mathbb{R}$ and
$g:[0,\infty)\times\mathbb{S}\times\mathbb{R}\to\mathbb{R}$. Let
$T\in(0,\infty)$ and $\lambda\in\mathcal{P}(\mathbb{S})$. Consider
the BSDE on \emph{$(\Omega,\{\mathcal{F}_{t}^{\lambda}\},\mathbb{P}_{\lambda})$}
\begin{equation}
\left\{ \begin{aligned}dY_{t} & =-g(t,Y_{t})dt-f(t,Y_{t},Z_{t})d\langle W\rangle_{t}+Z_{t}dW_{t},\;\;t\in[0,T),\\
Y_{T} & =\xi,
\end{aligned}
\right.\label{eq:-145}
\end{equation}
where $\xi$ is an $\mathcal{F}_{T}^{\lambda}\mbox{-}$measurable
random variable, and the coefficients $g,f$ satisfy the following
measurability condition:\medskip{}

\noindent $\hypertarget{M}{\text{(M)}}$\hspace{1em}$g:[0,\infty)\times\mathbb{R}\times\Omega\to\mathbb{R}$
and $f:[0,\infty)\times\mathbb{R}\times\mathbb{R}\times\Omega\to\mathbb{R}$
are measurable, and $t\mapsto g(t,y)$ and $t\mapsto f(t,y,z)$ are
$\{\mathcal{F}_{t}^{\lambda}\}$-adapted processes for all $y,z\in\mathbb{R}$.\medskip{}

\noindent As in the classical case, the BSDE (\ref{eq:-145}) is interpreted
as the corresponding (backward) integral equation
\begin{equation}
Y_{t}=\xi+\int_{t}^{T}g(r,Y_{r})dr+\int_{t}^{T}f(r,Y_{r},Z_{r})d\langle W\rangle_{r}-\int_{t}^{T}Z_{r}dW_{r},\;\;t\in[0,T]\;\;\mathbb{P}_{\lambda}\mbox{-a.s.},\label{eq:-2}
\end{equation}
provided that all integrals can be well-defined.
\begin{rem}
\label{rem:-11}Let us make some remarks on the different drift terms
in the BSDE (\ref{eq:-145}). The drift $g(t,Y_{t})dt$ corresponds
to Brownian motion $X_{t}$, while the drift $f(t,Y_{t},Z_{t})d\langle W\rangle_{t}$
corresponds to Brownian martingale $W_{t}$. As we shall see later
(Lemma \ref{lem:-33}), the processes $X_{t}$ and $W_{t}$ have singular
``speeds'', which is different from the situations on Euclidean
spaces and reveals the fractal nature of $\mathbb{S}$. We also note
that the process $Z_{t}$ does not appear in the drift $g(t,Y_{t})dt$.
This is because, intuitively, $Z_{t}$ represents the projection of
$Y_{t}$ on the ``$W_{t}$-direction'', which is orthogonal to the
projection of $Y_{t}$ on the ``$X_{t}$-direction''.
\end{rem}
We first introduce the Banach spaces for solutions of BSDEs. As BSDEs
with stochastic durations will also be considered, it is convenient
to define directly the spaces of processes with stochastic running
times.
\begin{defn}
\label{def:-1}Let $\tau$ be an $\{\mathcal{F}_{t}^{\lambda}\}$-stopping
time, and $\beta=(\beta_{0},\beta_{1})\in\mathbb{R}_{+}^{2}$. We
define $\mathcal{V}_{\lambda}^{\beta}[0,\tau]$ to be the Banach space
of all pairs $(y,z)$ of $\{\mathcal{F}_{t}^{\lambda}\}$-adapted
processes such that
\[
\begin{aligned}\left\Vert (y,z)\right\Vert _{\mathcal{V}_{\lambda}^{\beta}[0,\tau]}^{2}\triangleq\mathbb{E}_{\lambda}\Big[\;\sup_{0\le t\le\tau} & \Big(y_{t}^{2}e^{2\beta_{0}t+2\beta_{1}\langle W\rangle_{t}}+\int_{t}^{\tau}y_{r}^{2}e^{2\beta_{0}r+2\beta_{1}\langle W\rangle_{r}}dr\\
 & +\int_{t}^{\tau}\left(y_{r}^{2}+z_{r}^{2}\right)e^{2\beta_{0}r+2\beta_{1}\langle W\rangle_{r}}d\langle W\rangle_{r}\Big)\Big]<\infty.
\end{aligned}
\]
We shall simply write $\mathcal{V}_{x}^{\beta}[0,\tau]$ when $\lambda=\delta_{x}$
is the Dirac measure concentrated at $x\in\mathbb{S}$.
\end{defn}
\begin{defn}
\label{def:}Let $\beta\in\mathbb{R}_{+}^{2}$. We say that the BSDE
(\ref{eq:-145}) admits a \emph{solution $(Y,Z)$ in $\mathcal{V}_{\lambda}^{\beta}[0,T]$}
if $(Y,Z)\in\mathcal{V}_{\lambda}^{\beta}[0,T]$ and satisfies (\ref{eq:-2}).
The solution $(Y,Z)$ is said to be unique in $\mathcal{V}_{\lambda}^{\beta}[0,T]$
if $\|(Y-\bar{Y},Z-\bar{Z})\|_{\mathcal{V}_{\lambda}^{\beta}[0,T]}=0$
for any solution $(\bar{Y},\bar{Z})$ of (\ref{eq:-145}) in $\mathcal{V}_{\lambda}^{\beta}[0,T]$.
\end{defn}
We should point out that the uniqueness of solutions defined in Definition
\ref{def:} is the uniqueness of solutions in the space $\mathcal{V}_{\lambda}^{\beta}[0,T]$.
In fact, uniqueness can also be defined for solutions not necessarily
in $\mathcal{V}_{\lambda}^{\beta}[0,T]$. More precisely, we have
the following definition for uniqueness of solutions. 
\begin{defn}
\label{def:-4}The BSDE (\ref{eq:-145}) is said to \emph{admit at
most one solution}, if
\begin{equation}
Y_{t}=\bar{Y}_{t},\;t\in[0,T],\;\mbox{and}\;\int_{0}^{T}(Z_{r}-\bar{Z}_{r})^{2}d\langle W\rangle_{r}=0,\quad\mathbb{P}_{\lambda}\mbox{-a.s.}\label{eq:-15}
\end{equation}
for any two pairs $(Y,Z)$ and $(\bar{Y},\bar{Z})$ of $\{\mathcal{F}_{t}^{\lambda}\}$-adapted
processes satisfying (\ref{eq:-2}).
\end{defn}
\begin{thm}
\label{thm:-1-1}Let $\beta=(\beta_{0},\beta_{1})\in[1,\infty)^{2}$.
Suppose that
\begin{equation}
\mathbb{E}_{\lambda}\big(\xi^{2}e^{2\beta_{1}\langle W\rangle_{T}}\big)<\infty,\tag{A.1}\label{eq:-6}
\end{equation}
and that, for all $t\in[0,T)$ and all $y,\bar{y},z,\bar{z}\in\mathbb{R}$,
\begin{equation}
|g(t,y,\omega)-g(t,\bar{y},\omega)|\le\frac{K_{0}}{2}|y-\bar{y}|\quad\mathbb{P}_{\lambda}\mbox{-a.s.},\tag{A.2}\label{eq:-188}
\end{equation}
\begin{equation}
\left|f(t,y,z,\omega)-f(t,\bar{y},\bar{z},\omega)\right|\le\frac{K_{0}}{2}\left|y-\bar{y}\right|+K_{1}\,\left|z-\bar{z}\right|\;\;\mathbb{P}_{\lambda}\mbox{-a.s.},\tag{A.3}\label{eq:-189}
\end{equation}
where $K_{0},K_{1}>0$ are some constants.

(a) The BSDE (\ref{eq:-145}) admits at most one solution.

(b) If, in addition,
\begin{equation}
\mathbb{E}_{\lambda}\Big(\int_{0}^{T}g(r,0)^{2}e^{2\beta_{1}\langle W\rangle_{r}}dr\Big)<\infty,\;\mathbb{E}_{\lambda}\Big(\int_{0}^{T}f(r,0,0)^{2}e^{2\beta_{1}\langle W\rangle_{r}}d\langle W\rangle_{r}\Big)<\infty,\tag{A.4}\label{eq:-191}
\end{equation}
for sufficiently large $\beta_{0},\beta_{1}$ ($\beta_{i}>36K_{i}^{2},\,i=0,1$
will suffice), then (\ref{eq:-145}) admits a unique solution $(Y,Z)\in\mathcal{V}_{\lambda}^{\beta}[0,T]$.
Moreover,
\begin{equation}
\begin{aligned}\|(Y,Z)\|_{\mathcal{V}_{\lambda}^{\beta}[0,T]}^{2} & \le C\,\mathbb{E}_{\lambda}\Big(\xi^{2}e^{2\beta_{0}T+2\beta_{1}\langle W\rangle_{T}}+\int_{0}^{T}g(r,0)^{2}e^{2\beta_{0}r+2\beta_{1}\langle W\rangle_{r}}dr\\
 & \quad+\int_{0}^{T}f(r)^{2}e^{2\beta_{0}r+2\beta_{1}\langle W\rangle_{r}}d\langle W\rangle_{r}\Big),
\end{aligned}
\label{eq:---2}
\end{equation}
where $C>0$ is a constant depending only on $K_{0},K_{1},\beta$.
\end{thm}
\begin{rem}
\label{rem:}As we shall see in Theorem \ref{thm:-3-1} below, (Dirichlet)
terminal-boundary value problems for semi-linear parabolic PDEs correspond
to BSDEs with bounded stochastic durations. Results of Theorem \ref{thm:-1-1}
can be extended to such case with no essential difficulties. More
precisely, let $\tau$ be a bounded $\{\mathcal{F}_{t}^{\lambda}\}$-stopping
time and $\xi$ be an $\mathcal{F}_{\tau}^{\lambda}\mbox{-}$adaptive
random variable. The definition of solutions of the BSDE on \emph{$(\Omega,\{\mathcal{F}_{t}^{\lambda}\},\mathbb{P}_{\lambda})$}
\begin{equation}
\left\{ \begin{aligned}dY_{t} & =-g(t,Y_{t})dt-f(t,Y_{t},Z_{t})d\langle W\rangle_{t}+Z_{t}dW_{t},\;\;t\in[0,\tau),\\
Y_{\tau} & =\xi,
\end{aligned}
\right.\label{eq:-146}
\end{equation}
can be given by replacing $T$ by $\tau$ in Definition \ref{def:},
and the conclusions of Theorem \ref{thm:-1-1} also hold with $T$
replaced by $\tau$.\medskip{}
\end{rem}
We may also consider the BSDE (\ref{eq:-146}) with stochastic duration
$\tau$ which is not necessarily bounded, where $g,f$ satisfy the
measurability condition $\hyperlink{M}{\text{(M)}}$.
\begin{defn}
\label{def:-2}Let $\beta=(\beta_{0},\beta_{1})\in\mathbb{R}_{+}^{2}$.
We say that $(Y,Z)$ is a \emph{solution of }(\ref{eq:-146}) \emph{in
}$\mathcal{V}_{\lambda}^{\beta}[0,\tau]$ if $(Y,Z)\in\mathcal{V}_{\lambda}^{\beta}[0,\tau]$
and satisfies the following:
\begin{equation}
\begin{aligned}Y_{t\wedge\tau}=Y_{T\wedge\tau} & +\int_{t\wedge\tau}^{T\wedge\tau}g(r,Y_{r})dr+\int_{t\wedge\tau}^{T\wedge\tau}f(r,Y_{r},Z_{r})d\langle W\rangle_{r}\\
 & -\int_{t\wedge\tau}^{T\wedge\tau}Z_{r}dW_{r}\quad\mbox{for all}\;\;0\le t\le T<\infty\;\;\mathbb{P}_{\lambda}\mbox{-a.s.},
\end{aligned}
\label{eq:-147}
\end{equation}
and
\begin{equation}
\lim_{T\to\infty}\mathbb{E}_{\lambda}\Big(|Y_{T\wedge\tau}-\xi|^{2}e^{2\beta_{0}(T\wedge\tau)+2\beta_{1}\langle W\rangle_{T\wedge\tau}}\Big)=0.\label{eq:-151}
\end{equation}

The\emph{ }solution $(Y,Z)$ is said to be unique in $\mathcal{V}_{\lambda}^{\beta}[0,\tau]$
if $\|(Y-\bar{Y},Z-\bar{Z})\|_{\mathcal{V}_{\lambda}^{\beta}[0,\tau]}=0$
for any solution $(\bar{Y},\bar{Z})$ of (\ref{eq:-146}) in $\mathcal{V}_{\lambda}^{\beta}[0,\tau]$.
\end{defn}
\begin{rem}
\label{rem:-1}Clearly, if $\tau$ is bounded, then Definition \ref{def:-2}
coincides with Definition \ref{def:} with $T$ replaced by $\tau$.
(See Remark \ref{rem:}.)
\end{rem}
Similar to Definition \ref{def:}, we may discuss uniqueness of solutions
not necessarily in the space $\mathcal{V}_{\lambda}^{\beta}[0,\tau]$.
\begin{defn}
The BSDE (\ref{eq:-146}) is said to \emph{admit at most one solution},
if (\ref{eq:-15}) holds with $T$ replaced by $\tau$ for any two
pairs $(Y,Z)$ and $(\bar{Y},\bar{Z})$ of $\{\mathcal{F}_{t}^{\lambda}\}$-adapted
processes satisfying (\ref{eq:-147}) and (\ref{eq:-151}).
\end{defn}
\begin{thm}
\label{thm:-2-1}Let $\beta=(\beta_{0},\beta_{1})\in[1,\infty)^{2}$.
Suppose that (\ref{eq:-6})\textendash (\ref{eq:-189}) hold with
$T$ replaced by $\tau$:

\begin{equation}
\mathbb{E}_{\lambda}\big(\xi^{2}e^{2\beta_{0}\tau+2\beta_{1}\langle W\rangle_{\tau}}\big)<\infty,\tag{A'.1}\label{eq:-196}
\end{equation}
\begin{equation}
|g(t,y,\omega)-g(t,\bar{y},\omega)|\le\frac{K_{0}}{2}|y-\bar{y}|\quad\mathbb{P}_{\lambda}\mbox{-a.s.},\tag{A'.2}\label{eq:-173}
\end{equation}
\begin{equation}
|f(t,y,z,\omega)-f(t,\bar{y},\bar{z},\omega)|\le\frac{K_{0}}{2}|y-\bar{y}|+K_{1}|z-\bar{z}|\;\;\mathbb{P}_{\lambda}\mbox{-a.s.},\tag{A'.3}\label{eq:-174}
\end{equation}

(a) The BSDE (\ref{eq:-146}) admits at most one solution.

(b) If, in addition to (\ref{eq:-196})\textendash (\ref{eq:-174}),
\begin{equation}
(y-\bar{y})\big(g(t,y,\omega)-g(t,\bar{y},\omega)\big)\le-\kappa_{0}|y-\bar{y}|^{2},\tag{A'.4}\label{eq:-187}
\end{equation}
\begin{equation}
(y-\bar{y})\big(f(t,y,z,\omega)-f(t,\bar{y},z,\omega)\big)\le-\kappa_{1}|y-\bar{y}|^{2},\tag{A'.5}\label{eq:-172}
\end{equation}
for all $t\in[0,\tau(\omega))\;y,\bar{y},z,\bar{z}\in\mathbb{R}$,
and $\mathbb{P}_{\lambda}\mbox{-a.e.}\;\omega\in\Omega$, and
\begin{equation}
\mathbb{E}_{\lambda}\Big(\int_{0}^{\tau}g(t,0)^{2}e^{2\beta_{0}t+2\beta_{1}\langle W\rangle_{t}}dt\Big)<\infty,\;\mathbb{E}_{\lambda}\Big(\int_{0}^{\tau}f(t,0,0)^{2}e^{2\beta_{0}t+2\beta_{1}\langle W\rangle_{t}}d\langle W\rangle_{t}\Big)<\infty,\tag{A'.6}\label{eq:-183}
\end{equation}
where $\kappa_{0},\kappa_{1}\in\mathbb{R}$ are some constants satisfying
\begin{equation}
\beta_{0}-\kappa_{0}>0,\;\beta_{1}-\kappa_{1}+\frac{K_{1}^{2}}{2}>0.\tag{A'.7}\label{eq:-193}
\end{equation}
 Then BSDE (\ref{eq:-146}) admits a unique solution $(Y,Z)\in\mathcal{V}_{\lambda}^{\beta}[0,\tau]$,
and
\begin{equation}
\begin{aligned}\|(Y,Z)\|_{\mathcal{V}_{\lambda}^{\beta}[0,\tau]}^{2} & \le C\,\mathbb{E}_{\lambda}\Big(\xi^{2}e^{2\beta_{0}\tau+2\beta_{1}\langle W\rangle_{\tau}}+\int_{0}^{\tau}g(t,0)^{2}e^{2\beta_{0}t+2\beta_{1}\langle W\rangle_{t}}dt\\
 & \quad+\int_{0}^{\tau}f(t,0,0)^{2}e^{2\beta_{0}t+2\beta_{1}\langle W\rangle_{t}}d\langle W\rangle_{t}\Big),
\end{aligned}
\label{eq:-194}
\end{equation}
where $C>0$ is a constant depending only on $\beta_{0},\beta_{1},\kappa_{0},\kappa_{1},K_{0},K_{1}$.
\end{thm}
The last result of this paper addresses the representation of weak
solutions of semi-linear parabolic PDEs by those of corresponding
BSDEs. We first formulate the meaning of semi-linear parabolic PDEs
on $\mathbb{S}$. To this end, we introduce the following definition
of (weak) gradients of functions in $\mathcal{F}(\mathbb{S})$.
\begin{defn}
\label{def:-2-1}For any $u\in\mathcal{F}(\mathbb{S})$, in view of
Theorem \ref{thm:-}, we define the gradient $\nabla u$ of $u$ as
the unique element in $L^{2}(\mathbb{S};\nu)$ such that $M_{t}^{[u]}=\int_{0}^{t}\nabla u(X_{r})dW_{r},\;t\ge0$.
\end{defn}
\begin{rem}
\label{rem:-2}(i) Clearly, for the (reflected) Brownian motion on
$Q=[0,1]^{d}$, the definition of gradients given above coincides
with that of weak derivatives for functions in the Sobolev space $W^{1,2}(Q)$
of $L^{2}(Q)$ functions with distributional derivatives in $L^{2}(Q)$.

(ii) As pointed out in Section \ref{sec:}, there exist several definitions
of gradients on fractals in literature (for example, \citep{Ki93,Str00,Tep00,CS03,Hin10,HRT13,HT15,BK16})
specifically introduced to address different questions. Our definition
of gradients of functions is a slight variant of those in \citep{Ki93,Str00,Tep00},
and equivalent to those in \citep{Hin10,HRT13,BK16}. In \citep{HRT13}
and \citep{BK16}, gradient operators were introduced for Dirichlet
forms admitting (measure-valued) carré du champ operators $\Gamma(\cdot,\cdot)$,
which send functions $u,v\in\mathcal{F}(\mathbb{S})$ to $\Gamma(u,v)=\nu_{\langle u,v\rangle}$
(cf. Definition \ref{def:-5}). For the standard Dirichlet form on
the gasket, the isometry $\int_{\mathbb{S}}|\nabla u|^{2}\,d\nu=\int_{\mathbb{S}}d\Gamma(u,u)$
is valid as an immediate consequence of definitions, and the pointwise
equality $\nabla u=d\Gamma(u,\psi)/d\nu$ holds $\nu$-a.e. with a
suitable choice of $\psi\in\mathcal{F}(\mathbb{S})$,\footnote{The Radon-Nikodym derivative $(d\Gamma(u,\psi)/d\nu)(x)$ is denoted
by $\langle u,\psi\rangle_{\mathcal{H}_{x}}$ in \citep{HRT13,BK16}.} which is possible due to the results of \citep[Lemma 3.2, Theorem 4.1, and Theorem 4.2]{Hin13}
(see also \citep[Proposition 4.2]{BK16}).
\end{rem}
As an immediate consequence of Definition \ref{def:-2-1}, we have,
for any $u,\;v\in\mathcal{F}(\mathbb{S})$, that $\mathcal{E}(u,v)=\langle\nabla u,\nabla v\rangle_{\nu}$
and $\nu_{\langle u,v\rangle}=\nabla u\,\nabla v\cdot\nu$. Let $\Phi\in C^{1}(\mathbb{R}^{m})$
and $u_{1},\dots,u_{m}\in\mathcal{F}(\mathbb{S})$. Then $\Phi(\mathbf{u})\in\mathcal{F}(\mathbb{S})$
and $\nabla\Phi(\mathbf{u})=\sum_{i=1}^{m}\partial_{i}\Phi(\mathbf{u})\cdot\nabla u_{i}\;\;\nu\mbox{-a.e.}$,
where $\mathbf{u}=(u_{1},\dots,u_{m})$. In particular, for any $u,v\in\mathcal{F}(\mathbb{S})$,
$\nabla(uv)=\nabla u\cdot v+u\cdot\nabla v\;\;\nu\mbox{-a.e.}$

As an application of the theory of BSDEs, we consider semi-linear
parabolic equations on $\mathbb{S}$. On Euclidean spaces, the non-linear
parabolic PDE $\partial_{t}u+\mathcal{L}u=f(t,x,u,\nabla u)$ is interpreted
as $(\partial_{t}u+\mathcal{L}u)dx=f(t,x,u,\nabla u)dx$ when one
considers its weak solutions. The natural analogue of these PDEs on
$\mathbb{S}$ would be $(\partial_{t}u+\mathcal{L}u)\cdot\mu=f(t,x,u,\nabla u)\cdot\nu$.
However, for any $h\in\mathcal{F}(\mathbb{S})$, the gradient $\nabla h$
is defined only as a function in $L^{2}(\mathbb{S};\nu)$ and $\nu$
is singular to $\mu$, therefore, the above formulation does not have
a proper meaning for functions $u$ such that $u(t)\in\mathcal{F}(\mathbb{S})$
for $t\in[0,T)$. In view of this, it is more appropriate to formulate
(Dirichlet) terminal-boundary value problems for semi-linear parabolic
PDEs as those for measure equations.
\begin{defn}
For any $v\in L^{2}(\mu)$, let $\Vert v\Vert_{\mathcal{F}^{-1}}=\sup\big\{\langle u,v\rangle_{\mu}:u\in\mathcal{F}(\mathbb{S}),\,\mathcal{E}_{1}(u)\le1\big\}$.
The space $\mathcal{F}^{-1}(\mathbb{S})$ is defined to the $\Vert\cdot\Vert_{\mathcal{F}^{-1}}$-completion
of $L^{2}(\mu)$.
\end{defn}
\begin{defn}
Let $u\in L^{2}(0,T;\mathcal{F}(\mathbb{S}))$; that is, $\int_{0}^{T}\mathcal{E}_{1}(u(t))\,dt<\infty$.
The $\mathcal{F}(\mathbb{S})$-valued function $u$ is said to have\emph{
weak derivative in }$L^{2}(0,T;\mathcal{F}^{-1}(\mathbb{S}))$, if
there exists an $\mathcal{F}^{-1}(\mathbb{S})$-valued function $\partial_{t}u$
on $[0,T]$ such that 
\[
\Big(\int_{0}^{T}\Vert\partial_{t}u(t)\Vert_{\mathcal{F}^{-1}}^{2}\,dt\Big)^{1/2}<\infty\;\;\text{and}\;\int_{0}^{T}\langle u(t),\partial_{t}v(t)\rangle_{\mu}\,dt=-\int_{0}^{T}\langle\partial_{t}u(t),v(t)\rangle_{\mu}\,dt
\]
 for all $v\in C^{1}(0,T;\mathcal{F}(\mathbb{S}))$ with $v(0)=v(T)=0$.
\end{defn}
\begin{rem}
\label{rem:-10}Clearly, if $u\in L^{2}(0,T;\mathcal{F}(\mathbb{S}))\cap C^{1}(0,T;L^{2}(\mu))$,
then $u$ has a weak derivative in $L^{2}(0,T;\mathcal{F}^{-1}(\mathbb{S}))$.
\end{rem}
The following can be easily shown by a standard mollifier argument.
(See, for example, \citep[Theorem 3, Section 5.9]{Eva10}.)
\begin{lem}
\label{lem:-1}Suppose $u\in L^{2}(0,T;\mathcal{F}(\mathbb{S}))$
has weak derivative $\partial_{t}u$ in $L^{2}(0,T;\mathcal{F}^{-1}(\mathbb{S}))$.
Then $t\mapsto\Vert u(t)\Vert_{L^{2}(\mu)}^{2},\,t\in[0,T]$ is absolutely
continuous and $\frac{d}{dt}\Vert u(t)\Vert_{L^{2}(\mu)}^{2}=2\langle\partial_{t}u(t),u(t)\rangle_{\mu}\;\;\text{a.e. }t\in[0,T]$.
\end{lem}
As an application, we consider the following semi-linear parabolic
PDEs on $\mathbb{S}$. We should point out that there exist several
formulations of PDEs on fractals, and our formulation should be regarded
as an extension in this direction. (See, for example, \citep{Str06,HRT13,HT13}
and references therein.)
\begin{defn}
\label{def:-3-1}Let $\varphi\in C^{1,0}([0,T]\times\mathrm{V}_{0})$
and $\psi\in L^{2}(\mu)$. A function $u$ on $[0,T]\times\mathbb{S}$
is said to be a\emph{ weak solution} of the (Dirichlet) terminal-boundary
value problem
\begin{equation}
\left\{ \begin{aligned}(\partial_{t}u & +\mathcal{L}u)\cdot\mu=-g(t,x,u)\cdot\mu-f(t,x,u,\nabla u)\cdot\nu,\;\;(t,x)\in[0,T)\times\mathbb{S}\backslash\mathrm{V}_{0},\\
u(t & ,x)=\varphi(t,x)\;\text{on}\;[0,T)\times\mathrm{V}_{0},\;\ u(T)=\psi,
\end{aligned}
\right.\label{eq:}
\end{equation}
if the following are satisfied:

$\hypertarget{WS.1}{\text{(WS.1)}}$ $u\in C([0,T)\times\mathbb{S})\cap L^{2}(0,T;\mathcal{F}(\mathbb{S}))$
and $u$ has weak derivative $\partial_{t}u$ in $L^{2}(0,T;\mathcal{F}^{-1}(\mathbb{S}))$;

$\hypertarget{WS.2}{\text{(WS.2)}}$ for any $v\in\mathcal{F}(\mathbb{S}\backslash\mathrm{V}_{0})$,
\begin{equation}
\frac{d}{dt}\langle u(t),v\rangle_{\mu}-\mathcal{E}(u(t),v)=-\langle g(t,u(t)),v\rangle_{\mu}-\langle f(t,u(t),\nabla u(t),v\rangle_{\nu}\;\;\text{a.e.}\;t\in[0,T];\label{eq:-29}
\end{equation}

$\hypertarget{WS.3}{\text{(WS.3)}}$ $u(t,x)=\varphi(t,x)$ for each
$x\in\mathrm{V}_{0}$ and a.e. $t\in[0,T)$, and $\lim_{t\to T}u(t)=\psi$
in $L^{2}(\mu)$.
\end{defn}
\begin{rem}
\label{rem:-3}(i) We point out that the term $\langle f(t,u(t),\nabla u(t),v\rangle_{\nu}$
in (\ref{eq:-29}) is well-defined. In fact, $\nabla u$ is $\nu$-a.e.
defined and $u$ is pointwise defined by virtue of the fact $\mathcal{F}(\mathbb{S})\subseteq C(\mathbb{S})$,
which is, as pointed out in Section \ref{subsec:}, a corollary of
(\ref{eq:-40}).

\ (ii) The equation (\ref{eq:-29}) is well-posed by virtue of Lemma
\ref{lem:-1}.

(iii) In view of $\hyperlink{WS.2}{\text{(WS.2)}}$ and the singularity
of $\nu$ and $\mu$, we see that if $f\not=0$ then the PDE (\ref{eq:})
does not admit a solution $u$ such that $u\in C^{1,0}([0,T)\times\mathbb{S})$
and $u(t)\in\mathrm{Dom}(\mathcal{L})$, $t\in[0,T)$. This suggests
that the theory of PDEs on $\mathbb{S}$ is quite different from that
on $\mathbb{R}^{d}$.
\end{rem}
To construct weak solutions of the PDE (\ref{eq:}), a natural idea
is to show that the solution mapping of a related linear equation
is a contraction in some suitable Banach space, then iterate solutions
of this linear equation. However, difficulties arise immediately due
to the singularity of $\mu$ and $\nu$. To address this difficulty,
our idea is that, though calculus on fractals might be considerably
different from those on $\mathbb{R}^{d}$, stochastic calculus however
remains similar to its classical counterpart. Specifically, we have
following Feynman-Kac representation, which gives a BSDE approach
for semi-linear parabolic PDEs on $\mathbb{S}$.
\begin{thm}
\label{thm:-3-1}Let $\varphi\in C^{1,0}([0,T]\times\mathrm{V}_{0}),\;\psi\in L^{2}(\mu)$.
Let $g\in C([0,T]\times\mathbb{S}\times\mathbb{R})$ and $f\in C([0,T]\times\mathbb{S}\times\mathbb{R}^{2})$.
If the PDE (\ref{eq:}) admits a solution $u$, then, for each $s\in[0,T)$
and each $x\in\mathbb{S}$,
\[
(Y_{t}^{(s)},Z_{t}^{(s)})=(u(t+s,X_{t}),\nabla u(t+s,X_{t}))
\]
is the unique solution (in the sense of Theorem \ref{thm:-1-1}.(a))
of the BSDE
\begin{equation}
\left\{ \begin{aligned}dY_{t}^{(s)} & =-g(t+s,X_{t},Y_{t}^{(s)})dt\\
 & \quad-f(t+s,X_{t},Y_{t}^{(s)},Z_{t}^{(s)})d\langle W\rangle_{t}+Z_{t}^{(s)}dW_{t},\;\;t\in[0,\sigma^{(s)}),\\
Y_{\sigma^{(s)}}^{(s)} & =\Psi(\sigma^{(s)},X_{\sigma^{(s)}}),
\end{aligned}
\right.\label{eq:-3}
\end{equation}
on $\big(\Omega,\{\mathcal{F}_{t}\},\mathbb{P}_{x}\big)$ for each
$x\in\mathbb{S}$, where $\sigma^{(s)}=(T-s)\wedge\sigma_{\mathrm{V}_{0}},\;s\in[0,T]$,
and
\[
\Psi(t,x)=\begin{cases}
\varphi(t,x), & \mbox{if }(t,x)\in[0,T)\times\mathrm{V}_{0},\\
\psi(x), & \mbox{if }(t,x)\in\{T\}\times\mathbb{S}\backslash\mathrm{V}_{0}.
\end{cases}
\]
Moreover, the solution of (\ref{eq:}) is unique, and has the following
representation
\begin{equation}
u(t,x)=\mathbb{E}_{x}(Y_{0}^{(t)})\quad\text{for all}\ (t,x)\in[0,T)\times\mathbb{S}.\label{eq:-22}
\end{equation}
\end{thm}

\section{\label{sec:-2}Several results on Brownian martingale}

In this section, we collect several results, which will be needed
in later sections, on Brownian martingale. In particular, a representation
theorem for square-integrable martingale and a time-dependent Itô-Fukushima
decomposition are given. Moreover, we prove the exponential integrability
of the quadratic process of Brownian martingale, which is the main
technical result in this section.

\subsection{\label{subsec:-4}Martingale representations}

In this subsection, we consider representations for square-integrable
martingales adapted to filtrations induced by Brownian motion. The
following was proved in \citep[Theorem 3]{QY12} for general continuous
Hunt processes.
\begin{thm}
\label{thm:-1}Let $(\Omega,\{X_{t}\},\{\mathbb{P}_{x}\})$ be a continuous
Hunt process with state space $S$, and $\lambda\in\mathcal{P}(S)$.
Let $\mathcal{F}_{t}^{\lambda}$ be the $\mathbb{P}_{\lambda}$-completion
of $\sigma(X_{r}:r\le t)$. Then any local martingales on $(\Omega,\{\mathcal{F}_{t}^{\lambda}\},\mathbb{P}_{\lambda})$
are continuous. Furthermore, suppose that there exist a sub-algebra
$\mathcal{A}$ of $L^{2}(S;\lambda)\cap\mathcal{B}_{b}(S)$ and finitely
many continuous martingales $W^{1},\dots,W^{d}$ such that the following
are satisfied:\smallskip{}

(i) $\sigma(\mathcal{A})=\mathcal{B}(S)$ and $R_{\alpha}(\mathcal{A})\subseteq\mathcal{A}$
for each $\alpha>0$, where $R_{\alpha}$ denotes the $\alpha$-resolvent
of $\{X_{t}\}$;\smallskip{}

(ii) for any $f\in\mathcal{A}$ and any $\alpha>0$, there exist $\{\mathcal{F}_{t}^{\lambda}\}$-predictable
processes $f^{1},\dots,f^{d}$ such that $M_{t}^{\alpha,f}=\sum_{j=1}^{d}\int_{0}^{t}f^{j}\,dW_{r}^{j},\;t\ge0\;\;\mathbb{P}_{\lambda}$-a.s.,
where $M^{\alpha,f}$ is the martingale part of $R_{\alpha}u(X_{t})-R_{\alpha}u(X_{0})$;\smallskip{}

(iii) the matrix $\big(\langle W^{i},W^{j}\rangle_{t}\big)_{i,j}$
is strictly positive definite for all $t\ge0\;\;\mathbb{P}_{\lambda}$-a.s.\\
Then, for any square-integrable martingale $M$ on $(\Omega,\{\mathcal{F}_{t}^{\lambda}\},\mathbb{P}_{\lambda})$,
there uniquely exist $\{\mathcal{F}_{t}^{\lambda}\}$-predictable
processes $f^{1},\dots,f^{d}$ such that $M_{t}=M_{0}+\sum_{j=1}^{d}\int_{0}^{t}f^{j}(r)dW_{r}^{j},\;t\ge0\;\;\mathbb{P}_{\lambda}$-a.s.
\end{thm}
We now return to the setting of Brownian motion on the Sierpinski
gasket. Since $\mathcal{F}(\mathbb{S})\subseteq C(\mathbb{S})$, we
that $\mathcal{F}(\mathbb{S})$ is an algebra. Let $\lambda\in\mathcal{P}(\mathbb{S})$.
The assumptions (i) and (iii) in Theorem \ref{thm:-1} are clearly
satisfied with $\mathcal{A}=\mathcal{F}(\mathbb{S})$. The assumption
(ii) is also satisfied in view of Theorem \ref{thm:-}. Therefore,
we have the following representation theorem for square-integrable
martingales adapted to filtrations induced by the Brownian motion.
\begin{thm}
\label{thm:-2}Let $\lambda\in\mathcal{P}(\mathbb{S})$. For any square-integrable
martingale $M$ on $(\Omega,\{\mathcal{F}_{t}^{\lambda}\},\mathbb{P}_{\lambda})$,
there exists an $\{\mathcal{F}_{t}^{\lambda}\}$-predictable process
$f$ such that 
\[
M_{t}=M_{0}+\int_{0}^{t}f(r)dW_{r},\;\;t\ge0\quad\mathbb{P}_{\lambda}\mathit{-a.s.}
\]
The process $f$ is unique; that is, if $\bar{f}$ is another $\{\mathcal{F}_{t}^{\lambda}\}$-predictable
process satisfying the above, then $\mathbb{E}_{\lambda}\big[\int_{0}^{\infty}(f(r)-\bar{f}(r))^{2}d\langle W\rangle_{r}\big]=0$.
\end{thm}
We end this subsection by showing the uniqueness of decompositions
of semi-martingales $Y_{t}$ of the form
\begin{equation}
Y_{t}=Y_{0}+\int_{0}^{t}g(r)dr+\int_{0}^{t}f(r)d\langle W\rangle_{r}+M_{t},\;\;t\ge0,\label{eq:-14}
\end{equation}
 where $M$ is a martingale on $(\Omega,\{\mathcal{F}_{t}^{\mu}\},\mathbb{P}_{\mu})$.
\begin{lem}
\label{lem:-33}The Lebesgue-Stieltjes measure $d\langle W\rangle_{t}$
is singular to $dt\;\;\mathbb{P}_{\mu}$-a.s.
\end{lem}
\begin{proof}
Let $\mathcal{P}$ be the $\sigma$-filed of predictable sets in $[0,\infty)\times\Omega$.
Let $Q$ be the unique measure on $([0,\infty)\times\Omega,\mathcal{P})$
satisfying $Q\big(\llbracket\sigma,\tau\rrparenthesis\big)=\mathbb{E}_{\mu}\big(\langle W\rangle_{\tau}-\langle W\rangle_{\sigma}\big),\;\sigma,\tau\in\mathcal{T}_{p}$,
where $\mathcal{T}_{p}$ is the family of all $\{\mathcal{F}_{t}^{\mu}\}$-predictable
times, and $\llbracket\sigma,\tau\rrparenthesis=\{(t,\omega)\in[0,\infty)\times\Omega:\sigma(\omega)\le t<\tau(\omega)\}$.
By the Lebesgue decomposition, $Q=f\cdot(dt\times\mathbb{P}_{\mu})+Q_{s}$,
where $f\ge0$ is a predictable process and is $\sigma$-integrable
with respect to $dt\times\mathbb{P}_{\mu}$, and $Q_{s}$ is a $\sigma$-finite
positive measure singular to $dt\times\mathbb{P}_{\mu}$.

Note that, for any non-negative predictable process $g$,
\begin{equation}
\int_{[0,T)\times\Omega}g(r)Q(dr,d\omega)=\mathbb{E}_{\mu}\Big(\int_{0}^{T}g(r)d\langle W\rangle_{r}\Big),\label{eq:-205}
\end{equation}
which can be easily shown by the definition of $Q$ and a standard
monotone class argument. Now, let $B\in\mathcal{B}(\mathbb{S})$ such
that $\mu(B)=1=\nu(\mathbb{S}\backslash B)=1$. By $\mathbb{E}_{\mu}\big(\int_{0}^{T}1_{\mathbb{S}\backslash B}(X_{r})dr\big)=T\mu(\mathbb{S}\backslash B)=0$,
we see that $1_{B}(X_{t})=1,\;\mbox{a.e.}\;t\ge0,\;\mathbb{P}_{\mu}\mbox{-a.s.}$
Therefore, by (\ref{eq:-205}),
\[
\begin{aligned}\int_{0}^{T}\mathbb{E}_{\mu}(f(t))dt & =\int_{0}^{T}\mathbb{E}_{\mu}(f(t)1_{B}(X_{r}))dt\le\int_{[0,T)\times\Omega}1_{B}(X_{t}(\omega))Q(dt,d\omega)\\
 & =\mathbb{E}_{\mu}\Big(\int_{0}^{T}1_{B}(X_{t})d\langle W\rangle_{t}\Big)=T\nu(B)=0,
\end{aligned}
\]
which implies the conclusion of the lemma.
\end{proof}
\begin{cor}
Let $Y$ be a semi-martingale on $(\Omega,\{\mathcal{F}_{t}^{\mu}\},\mathbb{P}_{\mu})$
of the form (\ref{eq:-14}). Then the decomposition (\ref{eq:-14})
is unique; that is, if (\ref{eq:-14}) also holds with $g,f,M$ replaced
by $\bar{g},\bar{f},\bar{M}$, then $\mathbb{E}_{\mu}\big(\int_{0}^{\infty}|g(r)-\bar{g}(r)|dr\big)=\mathbb{E}_{\mu}\big(\int_{0}^{\infty}|f(r)-\bar{f}(r)|d\langle W\rangle_{r}\big)=0$.
\end{cor}
We shall need the following time-dependent Itô-Fukushima decomposition
(see also \citep{LZ90,FPS95} for similar results in different settings),
which follows from Theorem \ref{thm:-2} and the decomposition theorem
in \citep[Theorem 4.5]{Tru00} applied to the (non-symmetric) \emph{generalized
Dirichlet form} $(u,v)\mapsto\Lambda(u,v)+\int_{0}^{T}\mathcal{E}(u(t),v(t))\,dt$,
where
\[
\Lambda(u,v)=\left\{ \begin{aligned}\int_{0}^{T}\langle u(t),\partial_{t}v(t)\rangle_{\mu}\,dt,\;\;\text{if}\ u\in L^{2}(0,T;\mathcal{F}(\mathbb{S})),\,v\in L^{2}(0,T;\mathcal{F}(\mathbb{S}))\cap C^{1}(0,T;L^{2}(\mu)),\\
-\int_{0}^{T}\langle\partial_{t}u(t),v(t)\rangle_{\mu}\,dt,\;\;\text{if}\ u\in L^{2}(0,T;\mathcal{F}(\mathbb{S}))\cap C^{1}(0,T;L^{2}(\mu)),\,v\in L^{2}(0,T;\mathcal{F}(\mathbb{S})).
\end{aligned}
\right.
\]
(See, for example, \citep{Sta99,Tru00} and etc\emph{.}, for the theory
of generalized Dirichlet forms.)
\begin{lem}
\label{lem:}Suppose $u\in C([0,T)\times\mathbb{S})\cap L^{2}(0,T;\mathcal{F}(\mathbb{S}))$
and that $u$ has weak derivative $\partial_{t}u$ in $L^{2}(0,T;\mathcal{F}^{-1}(\mathbb{S}))$.
Then
\[
u(t,X_{t})=u(0,X_{0})+\int_{0}^{t}\nabla u(r,X_{r})\,dW_{r}+N_{t},\;\;t\in[0,T],
\]
where $N_{t}$ is a continuous processes with zero quadratic variation;
that is, for each $t\ge0$,
\[
\lim_{n\to\infty}\mathbb{E}_{\mu}\Big[\sum_{i=1}^{n}(N_{t_{i}}-N_{t_{i-1}})^{2}\Big]=0,
\]
where $t_{i}=\frac{i}{n}t,\;i=0,1,\dots,n$.
\end{lem}

\subsection{\label{subsec:-5}Exponential integrability of quadratic processes}

We now turn to the main technical result of this section, the exponential
integrability of the quadratic process $\left\langle W\right\rangle $,
which is a sufficient condition for the Girsanov theorem to hold and
is crucial for the existence of solutions of BSDEs. We shall need
the following heat kernel estimates. (See \citep[Theorem 5.3.1]{Ki01}.)
\begin{lem}
\label{lem:-29}(a) $\{X_{t}\}$ admits (jointly) continuous transition
kernels $p_{t}(x,y),\;t>0$, and there exists universal constants
$C_{\ast,1},C_{\ast,2}>0$ such that
\[
C_{\ast,1}\,t^{-d_{s}/2}\le p_{t}(x,y)\le C_{\ast,2}\,t^{-d_{s}/2},\;\;t\in(0,1],\;x,y\in\mathbb{S},
\]
where $d_{s}=2\log3/\log5$ is the \emph{spectral dimension} of $\{X_{t}\}$.

(b) $\{X_{t}^{0}\}$ admits (jointly) continuous transition kernels
$p_{t}^{0}(x,y),\;t>0$, and there are universal constants $C_{\ast,3},C_{\ast,4}>0$
such that
\[
C_{\ast,3}\,t^{-d_{s}/2}\le p_{t}^{0}(x,y)\le C_{\ast,4}\,t^{-d_{s}/2},\;\;t\in(0,1],\;x,y\in\mathbb{S}\backslash\mathrm{V}_{0}.
\]
\end{lem}
In general, $P_{t}f$ and $P_{t}^{0}f$ are well-defined only for
$f\in L^{2}(\mu)$. However, using the corresponding transition densities
and their continuities, these definitions can be extended to measures
(even possibly singular to $\mu$) as follows.
\begin{defn}
\label{def:-17}For each Radon measure $\lambda$ on $\mathbb{S}$
and each $t>0$, define
\begin{equation}
P_{t}\lambda(x)\triangleq\int_{\mathbb{S}}p_{t}(x,y)\lambda(dy),\quad x\in\mathbb{S}.\label{eq:-1}
\end{equation}
\end{defn}
\begin{rem}
\label{rem:-4}Clearly, $P_{t}\lambda\in C(\mathbb{S})$ and $|P_{t}\lambda|\le|\lambda|(\mathbb{S})<\infty$
for any Radon measure $\lambda$ on $\mathbb{S}$.
\end{rem}
\begin{lem}
\label{lem:-3}Let $A$ be a positive continuous additive functional
such that $\nu_{A}(\mathbb{S})<\infty$, where $\nu_{A}$ is the Revuz
measure of $A$. Then, for each $t>0$ and each $f\in\mathcal{B}_{b}([0,\infty)\times\mathbb{S})$,
\begin{equation}
\mathbb{E}_{x}\Big(\int_{0}^{t}f(r,X_{r})dA_{r}\Big)=\int_{0}^{t}P_{r}\big(f(r)\nu_{A}\big)(x)dr,\quad x\in\mathbb{S}.\label{eq:-4}
\end{equation}
Similarly, for positive continuous additive functional $A$ with respect
to $\{X_{t}^{0}\}$,
\begin{equation}
\mathbb{E}_{x}\Big(\int_{0}^{t}f(r,X_{r}^{0})dA_{r}\Big)=\int_{0}^{t}P_{r}^{0}\big(f(r)\nu_{A}\big)(x)dr,\quad x\in\mathbb{S}\backslash\mathrm{V}_{0}.\label{eq:-5}
\end{equation}
\end{lem}
\begin{proof}
By and the definition of $P_{r}(g\nu_{A})$ and a monotone class argument,
$\mathbb{E}_{\mu}\big(\big(\int_{0}^{t}f(r,X_{r})dA_{r}\big)h(X_{0})\big)=\int_{0}^{t}\big\langle P_{r}(f\nu_{A}),h\big\rangle_{\mu}dr$
for all $h\in\mathcal{B}_{b}(\mathbb{S})$, which implies (\ref{eq:-4})
in view of the continuity of both sides of (\ref{eq:-4}). The other
equality can be proved similarly.
\end{proof}
\begin{lem}
\label{lem:-28}Let $A^{(i)},\;i=1,\dots,n$ be positive continuous
additive functionals such that $\nu_{i}(\mathbb{S})<\infty$, where
$\nu_{i}$ is the Revuz measure of $A^{(i)},\;i=1,\dots,n$. Then,
for each $t>0$ and each $f_{i}\in\mathcal{B}_{b}([0,\infty)\times\mathbb{S}),\;i=1,\dots n$,
\begin{equation}
\begin{aligned}\mathbb{E}_{x} & \Big(f(X_{t})\int_{0<t_{1}<\cdots<t_{n}<t}f_{1}(t_{1},X_{t_{1}})\cdots f_{n}(t_{n},X_{t_{n}})dA_{t_{1}}^{(1)}\cdots dA_{t_{n}}^{(n)}\Big)\\
 & =\int_{0<t_{1}<\cdots<t_{n}<t}P_{t_{1}}\big(\nu_{1}f_{1}(t_{1})P_{t_{2}-t_{1}}(\cdots\nu_{n}f_{n}(t_{n})P_{t-t_{n}}f)\cdots\big)(x)dt_{1}\cdots dt_{n-1}dt_{n},\;\;x\in\mathbb{S}.
\end{aligned}
\label{eq:-23}
\end{equation}
\end{lem}
\begin{proof}
By (\ref{eq:-4}),
\[
\mathbb{E}_{x}\Big(f(X_{t})\int_{0}^{t}f_{1}(t_{1},X_{t_{1}})dA_{t_{1}}^{(1)}\Big)=\mathbb{E}_{x}\Big(\int_{0}^{t}f_{1}(t_{1},X_{t_{1}})\mathbb{E}_{x}\big(f(X_{t})\,|\,\mathcal{F}_{t_{1}}\big)dA_{t_{1}}^{(1)}\Big)=\int_{0}^{t}P_{t_{1}}\big(\nu_{1}f_{1}(t_{1})P_{t-t_{1}}f\big)(x)dt_{1}.
\]
This proves (\ref{eq:-23}) for $n=1$. The conclusion for a general
$n$follows readily from an induction argument.
\end{proof}
\begin{lem}
\label{lem:-30}Let $A$ be a positive continuous additive functional
such that $\nu_{A}(\mathbb{S})<\infty$, where $\nu_{A}$ is the Revuz
measure of $A$.

(a) For each $\beta>0$,
\begin{equation}
\sup_{x\in\mathbb{S}}\mathbb{E}_{x}\big(e^{\beta A_{t}}\big)\le\mathrm{E}_{\gamma_{s},1}\big[C_{\ast}\nu_{A}(\mathbb{S})\beta\max\{t,t^{\gamma_{s}}\}\big],\;\;t\ge0.\label{eq:-200}
\end{equation}
where $C_{\ast}>0$ denotes any (possibly different) universal constant,
$\gamma_{s}=1-d_{s}/2\in(0,1/2)$, $\mathrm{E}_{a,b}(z)=\sum_{p=0}^{\infty}\frac{z^{p}}{\Gamma(ap+b)},\;z\in\mathbb{C},\,a,b>0$
is the Mittag-Leffler function, and $\Gamma$ is the Gamma function.

(b) For each $f\in L_{+}^{1}(\mu)$ and $\beta>0$,
\[
\sup_{x\in\mathbb{S}}\mathbb{E}_{x}\big(f(X_{t})e^{\beta A_{t}}\big)\le\max\{1,t^{-d_{s}/2}\}\Vert f\Vert_{L^{1}(\mu)}\mathrm{E}_{\gamma_{s},\gamma_{s}}\big[C_{\ast}\nu_{A}(\mathbb{S})\beta\max\{t,t^{\gamma_{s}}\}\big],\;\;t>0.
\]
\end{lem}
\begin{rem}
To see the asymptotics of $\mathbb{E}_{x}(e^{\beta A_{t}})$ as $t\to0$,
we note that the Mittag-Leffler function $\mathrm{E}_{a,b}$ is an
entire function of order $1/a$, that is,
\[
\frac{1}{a}=\limsup_{r\to\infty}\frac{\log\log(\sup_{|z|\le r}|\mathrm{E}_{a,b}(z)|)}{\log r}.
\]
\end{rem}
\begin{proof}
(a) Suppose first that $t\in(0,1]$. For each $p\in\mathbb{N}_{+}$,
by Lemma \ref{lem:-28}, 
\[
\begin{aligned}\mathbb{E}_{x}\big(A_{t}^{p}\big) & =p!\,\mathbb{E}_{x}\Big(\int_{0<t_{1}<\cdots<t_{p}<t}dA_{t_{1}}\cdots dA_{t_{p}}\Big)\\
 & =p!\int_{0<t_{1}<\cdots<t_{p}<t}P_{t_{1}}(\nu_{A}P_{t_{2}-t_{1}}(\cdots\nu_{A}P_{t_{p}-t_{p-1}}(\mu_{A})\cdots))(x)dt_{1}\cdots dt_{p}.
\end{aligned}
\]
For each non-negative $f\in C(\mathbb{S})$, by Lemma \ref{lem:-29},
\begin{equation}
P_{r}(\nu_{A}f)(x)\le\Vert f\Vert_{L^{\infty}}P_{r}\nu_{A}\,(x)\le C_{\ast}\nu_{A}(\mathbb{S})r^{-d_{s}/2},\quad r>0,\;x\in\mathbb{S}\label{eq:-199}
\end{equation}
so that
\[
\Vert P_{t_{1}}(\nu_{A}P_{t_{2}-t_{1}}(\cdots\nu_{A}P_{t_{p}-t_{p-1}}(\nu_{A})\cdots))\Vert_{L^{\infty}}\le(C_{\ast}\nu_{A}(\mathbb{S}))^{p}[t_{1}(t_{2}-t_{1})\cdots(t_{p}-t_{p-1})]^{-d_{s}/2},
\]
which implies that
\begin{equation}
\begin{aligned}\frac{1}{p!}\mathbb{E}_{x}\left(\left\langle W\right\rangle _{t}^{p}\right) & \le(C_{\ast}\nu_{A}(\mathbb{S}))^{p}\int_{0<t_{1}<\cdots<t_{p}<t}[t_{1}(t_{2}-t_{1})\cdots(t_{p}-t_{p-1})]^{-d_{s}/2}dt_{1}\cdots dt_{p}\\
 & =(C_{\ast}\nu_{A}(\mathbb{S})t^{\gamma_{s}})^{p}\int_{0<\theta_{1}<\cdots<\theta_{p}<1}[\theta_{1}(\theta_{2}-\theta_{1})\cdots(\theta_{p}-\theta_{p-1})]^{-d_{s}/2}d\theta_{1}\cdots d\theta_{p}.
\end{aligned}
\label{eq:-26}
\end{equation}
Let
\[
\beta_{0}(\gamma_{1},\dots,\gamma_{p})=\int_{0<\theta_{1}<\cdots<\theta_{p}<1}\theta_{1}^{\gamma_{1}-1}(\theta_{2}-\theta_{1})^{\gamma_{2}-1}\cdots(\theta_{p}-\theta_{p-1})^{\gamma_{p}-1}d\theta_{1}\cdots d\theta_{p}.
\]
Then
\[
\begin{aligned} & \quad\beta_{0}(\gamma_{1},\dots,\gamma_{p})\\
 & =\int_{0<\theta_{2}<\cdots<\theta_{p}<1}\Big(\int_{0}^{\theta_{2}}\theta_{1}^{\gamma_{1}-1}(\theta_{2}-\theta_{1})^{\gamma_{2}-1}d\theta_{1}\Big)(\theta_{3}-\theta_{2})^{\gamma_{3}-1}\cdots(\theta_{p}-\theta_{p-1})^{\gamma_{p}-1}d\theta_{2}\cdots d\theta_{p}\\
 & =\int_{0<\theta_{2}<\cdots<\theta_{p}<1}\Big(\int_{0}^{1}\theta^{\gamma_{1}-1}(1-\theta)^{\gamma_{2}-1}d\theta\Big)\theta_{2}^{\gamma_{1}+\gamma_{2}-1}(\theta_{3}-\theta_{2})^{\gamma_{3}-1}\cdots(\theta_{p}-\theta_{p-1})^{\gamma_{p}-1}d\theta_{2}\cdots d\theta_{p}\\
\vphantom{\int_{0<\theta_{2}<\cdots<\theta_{p}<1}} & =\mathrm{B}(\gamma_{1},\gamma_{2})\beta_{0}(\gamma_{1}+\gamma_{2},\gamma_{3},\dots,\gamma_{p}),
\end{aligned}
\]
where $\mathrm{B}(\cdot,\cdot)$ is the Beta function. By induction,
we see that
\[
\begin{aligned}\beta_{0}(\gamma_{1},\dots,\gamma_{p}) & =\mathrm{B}(\gamma_{1},\gamma_{2})\mathrm{B}(\gamma_{1}+\gamma_{2},\gamma_{3})\cdots\mathrm{B}(\gamma_{1}+\cdots+\gamma_{p-1},\gamma_{p})\int_{0}^{1}\theta^{\gamma_{1}+\cdots+\gamma_{p}-1}d\theta\\
 & \vphantom{\int_{0}^{1}}=\frac{\mathrm{B}(\gamma_{1},\gamma_{2})\mathrm{B}(\gamma_{1}+\gamma_{2},\gamma_{3})\cdots\mathrm{B}(\gamma_{1}+\cdots+\gamma_{p-1},\gamma_{p})}{\gamma_{1}+\cdots+\gamma_{p}}.
\end{aligned}
\]
Therefore, by (\ref{eq:-26}),
\[
\begin{aligned}\frac{1}{p!}\mathbb{E}_{x}\big(A_{t}^{p}\big) & \le(C_{\ast}\nu_{A}(\mathbb{S})t^{\gamma_{s}})^{p}\beta(\gamma_{s},\dots,\gamma_{s})\\
 & =\frac{(C_{\ast}\nu_{A}(\mathbb{S})t^{\gamma_{s}})^{p}}{p\gamma_{s}}\prod_{i=1}^{p-1}\mathrm{B}(i\gamma_{s},\gamma_{s})=\frac{(C_{\ast}\nu_{A}(\mathbb{S})t^{\gamma_{s}})^{p}}{\Gamma(p\gamma_{s}+1)},
\end{aligned}
\]
for each $p\in\mathbb{N},\;t\in(0,1]$, which implies (\ref{eq:-200})
for $t\in[0,1]$.

For $t\in(1,\infty)$, we claim that
\begin{equation}
\sup_{x\in\mathbb{S}}\mathbb{E}_{x}\big(e^{\beta(A_{k+1}-A_{k})}\big)\le\mathrm{E}_{\gamma_{s},1}(C_{\ast}\nu_{A}(\mathbb{S})\beta),\quad\beta>0,\;k\in\mathbb{N}.\label{eq:---3}
\end{equation}
In fact, for any $f\in L^{1}(\mu)$,
\[
\begin{aligned}\Big|\int_{\mathbb{S}}f(x) & \mathbb{E}_{x}\big(e^{\beta(A_{k+1}-A_{k})}\big)\,\mu(dx)\Big|=\vphantom{\int_{\mathbb{S}}}\big|\mathbb{E}_{\mu}\big(f(X_{0})e^{\beta(A_{k+1}-A_{k})}\big)\big|\\
 & =\vphantom{\int_{\mathbb{S}}}\big|\mathbb{E}_{\mu}\big[f(X_{0})\mathbb{E}_{X_{k}}\big(e^{\beta A_{1}}\big)\big]\big|\le\vphantom{\int_{\mathbb{S}}}\mathrm{E}_{\gamma_{s},1}[C_{\ast}\nu_{A}(\mathbb{S})\beta]||f||_{L^{1}(\mu)},
\end{aligned}
\]
which implies that $\mathbb{E}_{x}\big(e^{\beta(A_{k+1}-A_{k})}\big)\le\mathrm{E}_{\gamma_{s},1}(C_{\ast}\Gamma(\gamma_{s})\nu_{A}(\mathbb{S})\beta),\;\mu\mbox{-a.e.}\;x\in\mathbb{S}$.
In particular,
\[
\sum_{p=0}^{N}\frac{1}{p!}\mathbb{E}_{x}\big((A_{k+1}-A_{k})^{p}\big)\le\mathrm{E}_{\gamma_{s},1}[C_{\ast}\nu_{A}(\mathbb{S})\beta],\quad N\in\mathbb{N},\;\mu\mbox{-a.e.}\;x\in\mathbb{S}.
\]
By Lemma \ref{lem:-28}, for each $N\in\mathbb{N}$, the function
$x\mapsto\sum_{p=0}^{N}\frac{1}{p!}\mathbb{E}_{x}\big((A_{k+1}-A_{k})^{p}\big)$
is continuous. Therefore, the above holds for each $N\in\mathbb{N}$
and each $x\in\mathbb{S}$. Letting $N\to\infty$ proves (\ref{eq:---3}).

For $n\in\mathbb{N}_{+}$ such that $n<t\le n+1\le2t$, by Hölder's
inequality and (\ref{eq:---3}), we have
\[
\mathbb{E}_{x}\big(e^{\beta A_{t}}\big)\le\prod_{i=0}^{n}\big[\mathbb{E}_{x}\big(e^{(n+1)\beta(A_{i+1}-A_{i})}\big)\big]^{1/(n+1)}\le\mathrm{E}_{\gamma_{s},1}[C_{\ast}\nu_{A}(\mathbb{S})\beta t].
\]
which proves (a).

(b) Suppose $t\in(0,1]$. By Lemma \ref{lem:-28},
\[
\mathbb{E}_{x}\big(f(X_{t})A_{t}^{p}\big)=p!\int_{0<t_{1}<\cdots<t_{p}<t}P_{t_{1}}(\nu_{A}P_{t_{2}-t_{1}}(\cdots\nu_{A}P_{t_{p}-t_{p-1}}(\nu_{A}P_{t-t_{p}}f)\cdots))(x)dt_{1}\cdots dt_{p}.
\]
According to (\ref{eq:-199}),
\[
\Vert P_{t_{1}}(\nu_{A}P_{t_{2}-t_{1}}(\cdots\nu_{A}P_{t_{p}-t_{p-1}}(\nu_{A}P_{t-t_{p}}f)\cdots))\Vert_{L^{\infty}}\le\Vert f\Vert_{L^{1}(\mu)}(C_{\ast}\nu_{A}(\mathbb{S}))^{p}[t_{1}(t_{2}-t_{1})\cdots(t-t_{p})]^{-d_{s}/2}
\]
so that
\[
\begin{aligned}\frac{1}{p!}\mathbb{E}_{x} & \big(f(X_{t})\langle W\rangle_{t}^{p}\big)\le(C_{\ast}\nu_{A}(\mathbb{S}))^{p}\int_{0<t_{1}<\cdots<t_{p}<t}[t_{1}(t_{2}-t_{1})\cdots(t-t_{p})]^{-d_{s}/2}dt_{1}\cdots dt_{p}\\
 & =t^{-d_{s}/2}\Vert f\Vert_{L^{1}(\mu)}(C_{\ast}\nu_{A}(\mathbb{S})t^{\gamma_{s}})^{p}\int_{0<\theta_{1}<\cdots<\theta_{p}<1}[\theta_{1}(\theta_{2}-\theta_{1})\cdots(1-\theta_{p})]^{-d_{s}/2}d\theta_{1}\cdots d\theta_{p}.
\end{aligned}
\]

The proof now proceeds similarly to that of (a).
\end{proof}
\begin{cor}
\label{cor:-8}For each $f\in L_{+}^{1}(\mu)$ and $\beta,t>0$, 
\[
\sup_{x\in\mathbb{S}}\mathbb{E}_{x}\big(f(X_{t})e^{\beta\langle W\rangle_{t}}\big)\le\max\{1,t^{-d_{s}/2}\}\Vert f\Vert_{L^{1}(\mu)}\mathrm{E}_{\gamma_{s},\gamma_{s}}[C_{\ast}\beta\max\{t,t^{\gamma_{s}}\}],
\]
where $C_{\ast}>0$ is a universal constant.
\end{cor}
\begin{cor}
\label{cor:}(a) For any $\beta\in\mathbb{R}$, $Z_{t}=e^{\beta W_{t}-\frac{1}{2}\beta^{2}\langle W\rangle_{t}},\;t\ge0$
is a martingale on $(\Omega,\{\mathcal{F}_{t}\},\mathbb{P}_{x})$
for each $x\in\mathbb{S}$.

(b) Let $x\in\mathbb{S},\;T>0$. Let $\{M_{t}\}_{t\ge0}$ be a continuous
local $\mathbb{P}_{x}\text{-}$martingale. Then $\tilde{M}_{t}=M_{t}-\langle M,W\rangle_{t},\;0\le t\le T$
is a continuous local $\mathbb{Q}_{x}$-martingale, where $\mathbb{Q}_{x}=Z_{T}\,\mathbb{P}_{x}$.
\end{cor}
We now prove the exponential integrability of $\langle W\rangle$
up to the hitting time $\sigma_{\mathrm{V}_{0}}$.
\begin{prop}
\label{lem:-26}There exists a $\beta_{0}>0$ such that $\sup_{x\in\mathbb{S}}\mathbb{E}_{x}\big(e^{\beta_{0}\langle W\rangle_{\sigma_{\mathrm{V}_{0}}}}\big)<\infty$.
\end{prop}
To prove Proposition \ref{lem:-26}, we need the following two lemmas.
\begin{lem}
\label{lem:-24}Let $h_{i}$ be the harmonic functions with boundary
values $h_{i}|_{\mathrm{V}_{0}}=1_{\{p_{i}\}},\,i=1,2,3$ respectively.
Then
\begin{equation}
\langle W\rangle=\frac{1}{3}\big(\langle M^{[h_{1}]}\rangle+\langle M^{[h_{2}]}\rangle+\langle M^{[h_{3}]}\rangle\big),\label{eq:-38}
\end{equation}
where $M^{[h_{i}]}$ are the martingale parts of $h_{i}(X_{t})-h_{i}(X_{0}),\,i=1,2,3$
respectively.
\end{lem}
\begin{proof}
Let $\mathbf{e}_{i}=1_{\{p_{i}\}},\,i=1,2,3$. For any $\omega\in\mathrm{W}_{\ast}$,
since $h_{i}\circ F_{[\omega]_{m}}$ is the harmonic function with
boundary value $\mathbf{A}_{[\omega]_{m}}\mathbf{e}_{i}$, it follows
from (\ref{eq:-32}) that
\[
\mathcal{E}(h_{i}\circ F_{[\omega]_{m}})=\frac{3}{2}\,\mathbf{e}_{i}^{t}\mathbf{A}_{[\omega]_{m}}^{t}\mathbf{P}\,\mathbf{A}_{[\omega]_{m}}\mathbf{e}_{i}=\frac{3}{2}\,\mathbf{e}_{i}^{t}\mathbf{Y}_{[\omega]_{m}}^{t}\mathbf{Y}_{[\omega]_{m}}\mathbf{e}_{i},\;\;i=1,2,3.
\]
Therefore, by the self-similar property (\ref{eq:-30}) of $\mathcal{E}$,
we deduce that
\[
\nu_{\langle h_{i}\rangle}(F_{[\omega]_{m}}(\mathbb{S}))=\Big(\frac{5}{3}\Big)^{m}\;\mathcal{E}(h_{i}\circ F_{[\omega]_{m}})=\frac{3}{2}\,\Big(\frac{5}{3}\Big)^{m}\;\mathbf{e}_{i}^{t}\mathbf{Y}_{[\omega]_{m}}^{t}\mathbf{Y}_{[\omega]_{m}}\mathbf{e}_{i},\;\;i=1,2,3.
\]
In particular,
\[
\sum_{i=1,2,3}\nu_{\langle h_{i}\rangle}(F_{[\omega]_{m}}(\mathbb{S}))=\frac{3}{2}\,\Big(\frac{5}{3}\Big)^{m}\;\mathrm{tr}(\mathbf{Y}_{[\omega]_{m}}^{t}\mathbf{Y}_{[\omega]_{m}})=3\nu(F_{[\omega]_{m}}(\mathbb{S})).
\]
Therefore, $\nu=\frac{1}{3}\sum_{i=1,2,3}\nu_{\langle h_{i}\rangle}$,
from which (\ref{eq:-38}) follows readily in view of the one-to-one
correspondence between Revuz measures and positive additive functionals.
\end{proof}
\begin{lem}
\label{lem:-25}Let $h\in\mathcal{F}(\mathbb{S})$ be a harmonic function.
Then, for each $x\in\mathbb{S}$, $M_{t}=h(X_{t\wedge\sigma_{\mathrm{V}_{0}}}),\;t\ge0$
is a BMO martingale on $(\Omega,\{\mathcal{F}_{t}\},\mathbb{P}_{x})$
and $\Vert M\Vert_{\mathrm{BMO}}\le\max_{\mathrm{V}_{0}}|h|$.
\end{lem}
\begin{proof}
We first show that $\{M_{t}\}$ is a martingale. Since $h$ is a harmonic
function, we have $h(x)=\mathbb{E}_{x}(h(X_{\sigma_{\mathrm{V}_{0}}})),\;x\in\mathbb{S}$.
Therefore
\[
\begin{aligned}M_{t} & =\mathbb{E}_{X_{t}}\big(h\big(X_{\sigma_{\mathrm{V}_{0}}}\big)\big)1_{\{t<\sigma_{\mathrm{V}_{0}}\}}+h\big(X_{\sigma_{\mathrm{V}_{0}}}\big)1_{\{t\ge\sigma_{\mathrm{V}_{0}}\}}\\
 & =\mathbb{E}_{x}\big(h\big(X_{t+\sigma_{\mathrm{V}_{0}}\circ\theta_{t}}\big)1_{\{t<\sigma_{\mathrm{V}_{0}}\}}\big|\mathcal{F}_{t}\big)+h\big(X_{\sigma_{\mathrm{V}_{0}}}\big)1_{\{t\ge\sigma_{\mathrm{V}_{0}}\}}\\
 & =\mathbb{E}_{x}\big(h\big(X_{\sigma_{\mathrm{V}_{0}}}\big)1_{\{t<\sigma_{\mathrm{V}_{0}}\}}\big|\mathcal{F}_{t}\big)+h\big(X_{\sigma_{\mathrm{V}_{0}}}\big)1_{\{t\ge\sigma_{\mathrm{V}_{0}}\}}.
\end{aligned}
\]
Note that $h(X_{\sigma_{\mathrm{V}_{0}}})1_{\{t\ge\sigma_{\mathrm{V}_{0}}\}}\in\mathcal{F}_{t}$.
We see that $M_{t}=\mathbb{E}_{x}(h(X_{\sigma_{\mathrm{V}_{0}}})\big|\mathcal{F}_{t}),\;t\ge0$.
This implies that $\{M_{t}\}$ is a martingale on $\big(\Omega,\{\mathcal{F}_{t}\},\mathbb{P}_{x}\big)$.

By the maximum principle for harmonic functions (see \citep[Theorem 3.2.5]{Ki01}),
$|M_{t}|\le\max_{\mathrm{V}_{0}}|h|$. Note that $M_{t}=M_{t\wedge\sigma_{\mathrm{V}_{0}}},\;t\ge0$.
We see that
\[
\mathbb{E}_{x}\big(\langle M\rangle_{\infty}-\langle M\rangle_{\tau}\big)=\mathbb{E}_{x}\big(M_{\sigma_{\mathrm{V}_{0}}}^{2}-M_{\tau\wedge\sigma_{\mathrm{V}_{0}}}^{2}\big)\le\max_{\mathrm{V}_{0}}|h|^{2},
\]
which implies $\Vert M\Vert_{\mathrm{BMO}}\le\max_{\mathrm{V}_{0}}|h|$.
\end{proof}
\begin{proof}[Proof of Proposition \ref{lem:-26}]
The conclusion follows immediately from Lemma \ref{lem:-25}, Lemma
\ref{lem:-24} and the John-Nirenberg inequality for BMO martingales.
\end{proof}

\section{\label{sec:-3}BSDEs driven by Brownian martingale}

\subsection{\label{subsec:-2}BSDEs with deterministic durations}

In this subsection, we prove the existence and uniqueness of solutions
of the BSDE (\ref{eq:-145}). Results and proofs in this subsection
are also valid if $T$ is replaced by a bounded $\{\mathcal{F}_{t}^{\lambda}\}\mbox{-}$stopping
time $\tau$.

We start with the simple case where $g,f$ do not depend on $y$ or
$z$; that is,
\begin{equation}
\left\{ \begin{aligned}dY_{t} & =-g(t)dt-f(t)d\langle W\rangle_{t}+Z_{t}dW_{t},\;\;t\in[0,T),\\
Y_{T} & =\xi,
\end{aligned}
\right.\label{eq:-170}
\end{equation}
where $g(t)=g(t,\omega)$ and $f(t)=f(t,\omega)$ are $\{\mathcal{F}_{t}^{\lambda}\}$-adapted
processes.
\begin{lem}
\label{lem:-2}Let $\beta=(\beta_{0},\beta_{1})\in[1,\infty)^{2}$,
and let $\xi\in\mathcal{F}_{T}^{\lambda}$ satisfy (\ref{eq:-6}).
Suppose that $g(t),f(t)$ are $\{\mathcal{F}_{t}^{\lambda}\}$-adapted
processes such that
\[
\mathbb{E}_{\lambda}\Big(\int_{0}^{T}g(r)^{2}e^{2\beta_{1}\langle W\rangle_{r}}dr\Big)<\infty,\;\mathbb{E}_{\lambda}\Big(\int_{0}^{T}f(r)^{2}e^{2\beta_{1}\langle W\rangle_{r}}d\langle M^{1}\rangle_{r}\Big)<\infty.
\]
Then the BSDE (\ref{eq:-170}) admits a unique solution $(Y,Z)$ in
$\mathcal{V}_{\lambda}^{\beta}[0,T]$. Moreover,
\begin{equation}
\begin{aligned}\left\Vert (Y,Z)\right\Vert _{\mathcal{V}_{\lambda}^{\beta}[0,T]}^{2}\le10\,\Big[ & \mathbb{E}_{\lambda}\big(\xi^{2}e^{2\beta_{0}T+2\beta_{1}\langle W\rangle_{T}}\big)+\mathbb{E}_{\lambda}\Big(\int_{0}^{T}g(r)^{2}e^{2\beta_{0}r+2\beta_{1}\langle W\rangle_{r}}dr\Big)\\
 & +\mathbb{E}_{\lambda}\Big(\int_{0}^{T}f(r)^{2}e^{2\beta_{0}r+2\beta_{1}\langle W\rangle_{r}}d\langle W\rangle_{r}\Big)\Big].
\end{aligned}
\label{eq:-21}
\end{equation}
\end{lem}
\begin{proof}
Let
\[
Y_{t}=\mathbb{E}_{\lambda}\Big(\xi+\int_{t}^{T}g(r)dr+\int_{t}^{T}f(r)d\langle W\rangle_{r}\,\Big|\,\mathcal{F}_{t}\Big),\;\;t\in[0,T].
\]
Then $Y_{T}=\xi$, and 
\[
Y_{t}+\int_{0}^{t}g(r)dr+\int_{0}^{t}f(r)d\langle W\rangle_{r}=\mathbb{E}_{\lambda}\Big(\xi+\int_{0}^{T}g(r)dr+\int_{0}^{T}f(r)d\langle W\rangle_{r}\,\Big|\,\mathcal{F}_{t}\Big)
\]
is a martingale on $(\Omega,\{\mathcal{F}_{t}^{\lambda}\},\mathbb{P}_{\lambda})$.
By Theorem \ref{thm:-2}, there exists a unique $\{\mathcal{F}_{t}^{\lambda}\}$-predictable
process $Z$ such that
\[
Y_{t}-Y_{0}+\int_{0}^{t}g(r)dr+\int_{0}^{t}f(r)d\langle W\rangle_{r}=\int_{0}^{t}Z_{r}dW_{r},\;\;t\in[0,T]\;\;\mathbb{P}_{\lambda}\mbox{-a.s.}
\]
which, together with $Y_{T}=\xi$, implies that
\[
Y_{t}=\xi+\int_{t}^{T}g(r)dr+\int_{t}^{T}f(r)d\langle W\rangle_{r}-\int_{t}^{T}Z_{r}dW_{r},\;\;t\in[0,T]\;\;\mathbb{P}_{\lambda}\mbox{-a.s.}
\]
Therefore, $(Y,Z)$ is a solution of the BSDE (\ref{eq:-170}).

We now turn to the proof of (\ref{eq:-21}), from which the uniqueness
of the solution in $\mathcal{V}_{\lambda}^{\beta}[0,T]$ follows immediately.
Denote $\mathrm{e}_{t}=\exp\left(2\beta_{0}t+2\beta_{1}\langle W\rangle_{t}\right)$.
By Itô's formula,
\[
\begin{aligned}Y_{t}^{2}\mathrm{e}_{t}= & \xi^{2}\mathrm{e}_{T}+\int_{t}^{T}(-2\beta_{0}Y_{r}^{2}+2Y_{r}g(r))\mathrm{e}_{r}dr\\
 & +\int_{t}^{T}(-2\beta_{1}Y_{r}^{2}+2Y_{r}f(r)-Z_{r}^{2})\mathrm{e}_{r}d\langle W\rangle_{r}-2\int_{t}^{T}Y_{r}Z_{r}\mathrm{e}_{r}dW_{r},
\end{aligned}
\]
which implies that
\begin{equation}
\begin{aligned} & Y_{t}^{2}\mathrm{e}_{t}+2\beta_{0}\int_{t}^{T}Y_{r}^{2}\mathrm{e}_{r}dr+\int_{t}^{T}(2\beta_{1}Y_{r}^{2}+Z_{r}^{2})\mathrm{e}_{r}d\langle W\rangle_{r}\\
 & \le\xi^{2}\mathrm{e}_{T}+\int_{t}^{T}\Big(\beta_{0}Y_{r}^{2}+\frac{1}{\beta_{0}}g(r)^{2}\Big)\mathrm{e}_{r}dr+\int_{t}^{T}\Big(\beta_{1}Y_{r}^{2}+\frac{1}{\beta_{1}}f(r)^{2}\Big)\mathrm{e}_{r}d\langle W\rangle_{r}-2\int_{t}^{T}Y_{r}Z_{r}\mathrm{e}_{r}dW_{r}.
\end{aligned}
\label{eq:-181}
\end{equation}
Taking expectations on both sides of the above and using a localization
argument give
\begin{equation}
\mathbb{E}_{\lambda}\Big(\int_{0}^{T}Z_{r}^{2}\mathrm{e}_{r}d\langle W\rangle_{r}\Big)\le\mathbb{E}_{\lambda}\Big(\xi^{2}\mathrm{e}_{T}+\frac{1}{\beta_{0}}\int_{t}^{T}g(r)^{2}\mathrm{e}_{r}dr+\frac{1}{\beta_{1}}\int_{t}^{T}f(r)^{2}\mathrm{e}_{r}d\langle W\rangle_{r}\Big).\label{eq:-13}
\end{equation}

By (\ref{eq:-181}) again, 
\begin{equation}
\begin{aligned} & Y_{t}^{2}\mathrm{e}_{t}+\beta_{0}\int_{t}^{T}Y_{r}^{2}\mathrm{e}_{r}dr+\int_{t}^{T}\left(\beta_{1}Y_{r}^{2}+Z_{r}^{2}\right)\mathrm{e}_{r}d\langle W\rangle_{r}\\
 & \le\xi^{2}\mathrm{e}_{T}+\frac{1}{\beta_{0}}\int_{t}^{T}g(r)^{2}\mathrm{e}_{r}dr+\frac{1}{\beta_{1}}\int_{t}^{T}f(r)^{2}\mathrm{e}_{r}d\langle W\rangle_{r}+2\,\Big|\int_{t}^{T}Y_{r}Z_{r}\mathrm{e}_{r}dW_{r}\Big|.
\end{aligned}
\label{eq:-18}
\end{equation}
By Doob's maximal inequality,
\begin{equation}
\begin{aligned}\mathbb{E}_{\lambda}\Big(\,\sup_{0\le t\le T} & \left|\int_{t}^{T}Y_{r}Z_{r}\mathrm{e}_{r}dW_{r}\right|\Big)\le2\,\mathbb{E}_{\lambda}\Big(\,\sup_{0\le t\le T}\left|\int_{0}^{t}Y_{r}Z_{r}\mathrm{e}_{r}dW_{r}\right|\Big)\\
 & \le2\,\mathbb{E}_{\lambda}\Big[\Big(\int_{0}^{T}Y_{r}^{2}Z_{r}^{2}\mathrm{e}_{r}^{2}d\langle W\rangle_{r}\Big)^{1/2}\Big]\le\frac{1}{4}\Vert(Y,Z)\Vert_{\mathcal{V}_{\lambda}^{\beta}[0,T]}^{2}+4\,\mathbb{E}_{\lambda}\Big(\int_{0}^{T}Z_{r}^{2}\mathrm{e}_{r}d\langle W\rangle_{r}\Big),
\end{aligned}
\label{eq:-19}
\end{equation}
By the above, (\ref{eq:-18}) and (\ref{eq:-13}), we obtain that
\[
\begin{aligned}\big\Vert(Y,Z)\big\Vert_{\mathcal{V}_{\lambda}^{\beta}[0,T]}^{2} & \le\frac{1}{2}\Vert(Y,Z)\Vert_{\mathcal{V}_{\lambda}^{\beta}[0,T]}^{2}+5\,\mathbb{E}_{\lambda}\Big(\xi^{2}\mathrm{e}_{T}+\frac{1}{\beta_{0}}\int_{0}^{T}g(r)^{2}\mathrm{e}_{r}dr+\frac{1}{\beta_{1}}\int_{0}^{T}f(r)^{2}\mathrm{e}_{r}d\langle W\rangle_{r}\Big).\end{aligned}
\]
which, together with a localization argument if necessary, completes
the proof.
\end{proof}
The following a priori estimate is crucial to the proof of Theorem
\ref{thm:-1-1}.
\begin{lem}
\label{thm:-4}Let $\beta=(\beta_{0},\beta_{2})\in[1,\infty)^{2}$.
Suppose that $\xi,g,f$ satisfy (\ref{eq:-6})-{}-(\ref{eq:-191}).
In view of Lemma \ref{lem:-2}, the BSDE 
\[
\left\{ \begin{aligned}dY_{t} & =-g(t,y_{t})dt-f(t,y_{t},z_{t})d\langle W\rangle_{t}+Z_{t}dW_{t},\;\;t\in[0,T),\\
Y_{T} & =\xi,
\end{aligned}
\right.
\]
admits a unique solution $(Y,Z)$ in $\mathcal{V}_{\lambda}^{\beta}[0,T]$
for any $(y,z)\in\mathcal{V}_{\lambda}^{\beta}[0,T]$. Let $F:\mathcal{V}_{\lambda}^{\beta}[0,T]\to\mathcal{V}_{\lambda}^{\beta}[0,T]$
be the solution map $(y,z)\mapsto(Y,Z)$. If $(\bar{y},\bar{z})\in\mathcal{V}_{\lambda}^{\beta}[0,T]$
and $(\bar{Y},\bar{Z})=F(\bar{y},\bar{z})$, then
\[
\left\Vert (\hat{Y},\hat{Z})\right\Vert _{\mathcal{V}_{\lambda}^{\beta}[0,T]}\le3\sqrt{2}K_{\beta}\left\Vert (\hat{y},\hat{z})\right\Vert _{\mathcal{V}_{\lambda}^{\beta}[0,T]},
\]
where $\hat{\eta}=\eta-\bar{\eta}$ for $\eta=y,z,Y,Z$, and 
\begin{equation}
K_{\beta}^{2}=\frac{K_{0}^{2}}{\beta_{0}}+\frac{K_{1}^{2}}{\beta_{1}}.\label{eq:-20}
\end{equation}
Moreover, $F$ is a $|\negthinspace|\cdot|\negthinspace|_{\mathcal{V}_{\lambda}^{\beta}[0,T]}\mbox{-}$contraction
when $\beta_{0},\beta_{1}$ are sufficiently large ($\beta_{i}\ge36K_{i}^{2},\,i=0,1$
will suffice).
\end{lem}
\begin{proof}
Let $\hat{g}_{t}=g(t,y_{t})-g(t,\bar{y}_{t}),\;\hat{f}_{t}=f(t,y_{t},z_{t})-f(t,\bar{y}_{t},\bar{z}_{t})$.
Then $|\hat{g}_{t}|\le\frac{K_{0}}{2}\,|\hat{y}_{t}|,\;|\hat{f}_{t}|\le\frac{K_{0}}{2}\left|\hat{y}_{t}\right|+K_{1}\left|\hat{z}_{t}\right|$,
and
\[
d\hat{Y}_{t}=-\hat{g}_{t}dt-\hat{f}_{t}d\langle W\rangle_{t}+\hat{Z}_{t}dW_{t},\quad\hat{Y}_{T}=0.
\]
Let $\mathrm{e}_{t}=\exp\left(2\beta_{0}t+2\beta_{1}\langle W\rangle_{t}\right)$.
Similar to the derivation of (\ref{eq:-18}), we have
\begin{equation}
\begin{aligned}\hat{Y}_{t}^{2}\mathrm{e}_{t} & +\beta_{0}\int_{t}^{T}\hat{Y}_{r}^{2}\mathrm{e}_{r}\,dr+\int_{t}^{T}(\beta_{1}\hat{Y}_{r}^{2}+\hat{Z}_{r}^{2})\mathrm{e}_{r}\,d\langle W_{r}\rangle_{r}\\
 & \le K_{\beta}^{2}\int_{t}^{T}\hat{y}_{r}^{2}\mathrm{e}_{r}\,dr+K_{\beta}^{2}\int_{t}^{T}\hat{z}_{r}^{2}\mathrm{e}_{r}\,d\langle W\rangle_{r}-2\int_{t}^{T}\hat{Y}_{r}\hat{Z}_{r}\mathrm{e}_{r}\,dW_{r},
\end{aligned}
\label{eq:-24}
\end{equation}
where $K_{\beta}>0$ is given by (\ref{eq:-20}). Proceeding as in
the proof of Lemma \ref{lem:-2}, we obtain
\[
\left\Vert (\hat{Y},\hat{Z})\right\Vert _{\mathcal{V}_{\lambda}^{\beta}[0,T]}^{2}\le\frac{1}{2}\left\Vert (\hat{Y},\hat{Z})\right\Vert _{\mathcal{V}_{\lambda}^{\beta}[0,T]}^{2}+9K_{\beta}^{2}\left\Vert (\hat{y},\hat{z})\right\Vert _{\mathcal{V}_{\lambda}^{\beta}[0,T]}^{2},
\]
which completes the proof.
\end{proof}
We are now in a position to prove Theorem \ref{thm:-1-1}.
\begin{proof}[Proof of Theorem \ref{thm:-1-1}]
(a) Suppose that $(Y,Z)$ and $(\bar{Y},\bar{Z})$ are two pairs
of $\{\mathcal{F}_{t}^{\lambda}\}\mbox{-}$adapted processes satisfying
(\ref{eq:-2}). Denote $\hat{\eta}=\eta-\bar{\eta}$ for $\eta=y,z,Y,Z$,
and let $\hat{g}_{t}=g(t,Y_{t})-g(t,\bar{Y}_{t})$, $\hat{f}_{t}=f(t,Y_{t},Z_{t})-f(t,\bar{Y}_{t},\bar{Z}_{t})$.
Similar to the derivation of (\ref{eq:-24}), we have
\[
\begin{aligned}\hat{Y}_{t}^{2}\mathrm{e}_{t} & +\beta_{0}\int_{t}^{T}\hat{Y}_{r}^{2}\mathrm{e}_{r}\,dr+\int_{t}^{T}(\beta_{1}\hat{Y}_{r}^{2}+\hat{Z}_{r}^{2})\mathrm{e}_{r}\,d\langle W_{r}\rangle_{r}\\
 & \le K_{\beta}^{2}\int_{t}^{T}\hat{Y}_{r}^{2}\mathrm{e}_{r}\,dr+K_{\beta}^{2}\int_{t}^{T}\hat{Z}_{r}^{2}\mathrm{e}_{r}\,d\langle W\rangle_{r}-2\int_{t}^{T}\hat{Y}_{r}\hat{Z}_{r}\mathrm{e}_{r}\,dW_{r},
\end{aligned}
\]
where $K_{\beta}>0$ is given by (\ref{eq:-20}). Setting $\beta_{i}=4K_{i}^{2},\,i=0,1$
in the above gives 
\[
\mathbb{E}_{\lambda}\Big(\hat{Y}_{t}^{2}\mathrm{e}_{t}+\frac{1}{2}\int_{t}^{T}\hat{Y}_{r}^{2}\mathrm{e}_{r}dr+\int_{t}^{T}\Big(\frac{1}{2}\hat{Y}_{r}^{2}+\hat{Z}_{r}^{2}\Big)\mathrm{e}_{r}d\langle W\rangle_{r}\Big)\le0,\;\;t\in[0,T],
\]
which completes the proof of (a).

(b) Suppose, in addition, that (\ref{eq:-191}) is satisfied. Let
$(Y^{(0)},Z^{(0)})=(0,0)$. By virtue of Lemma \ref{thm:-4}, define
inductively $(Y^{(n)},Z^{(n)})\in\mathcal{V}_{\lambda}^{\beta}[0,T],\;n\in\mathbb{N}_{+}$
to be the unique solution in $\mathcal{V}_{\lambda}^{\beta}[0,T]$
of the BSDE
\[
\left\{ \begin{aligned}dY_{t}^{(n)} & =-g(t,Y_{t}^{(n-1)})dt-f(t,Y_{t}^{(n-1)},Z_{t}^{(n-1)})d\langle W\rangle_{t}+Z_{t}^{(n)}dW_{t},\;\;t\in[0,T),\\
Y_{T}^{(n)} & =\xi.
\end{aligned}
\right.
\]
By Lemma \ref{thm:-4},
\begin{equation}
\begin{aligned}\big\Vert(Y^{(n+1)} & -Y^{(n)},Z^{(n+1)}-Z^{(n)})\big\Vert_{\mathcal{V}_{\lambda}^{\beta}[0,T]}\\
 & \le K_{\beta}\big\Vert(Y^{(n)}-Y^{(n-1)},Z^{(n)}-Z^{(n-1)})\big\Vert_{\mathcal{V}_{\lambda}^{\beta}[0,T]},\;\;n\in\mathbb{N}_{+},
\end{aligned}
\label{eq:--}
\end{equation}
where $K_{\beta}>0$ is given by (\ref{eq:-20}). By Lemma \ref{lem:-2},
\begin{equation}
\begin{aligned}\big\Vert(Y^{(1)},Z^{(1)})\big\Vert_{\mathcal{V}_{\lambda}^{\beta}[0,T]}^{2} & \le10\,\Big[\mathbb{E}_{\lambda}\big(\xi^{2}e^{2\beta_{0}T+2\beta_{1}\langle W\rangle_{T}}\big)+\mathbb{E}_{\lambda}\Big(\int_{0}^{T}g(r,0)^{2}e^{2\beta_{0}r+2\beta_{1}\langle W\rangle_{r}}dr\Big)\\
 & \quad+\mathbb{E}_{\lambda}\Big(\int_{0}^{T}f(r)^{2}e^{2\beta_{0}r+2\beta_{1}\langle W\rangle_{r}}d\langle W\rangle_{r}\Big)\Big].
\end{aligned}
\label{eq:---1}
\end{equation}

Choose $\beta_{0},\beta_{1}>0$ sufficiently large so that $K_{\beta}<1$
($\beta_{i}>36K_{i}^{2},\,i=0,1$ will suffice). By (\ref{eq:--})
and (\ref{eq:---1}), we conclude that $(Y^{(n)},Z^{(n)}),\;n\in\mathbb{N}_{+}$
is a Cauchy sequence in $\mathcal{V}_{\lambda}^{\beta}[0,T]$. Moreover,
$\lim_{n\to\infty}\Vert(Y^{(n)}-Y,Z^{(n)}-Z)\Vert_{\mathcal{V}_{\lambda}^{\beta}[0,T]}=0$
for some $(Y,Z)\in\mathcal{V}_{\lambda}^{\beta}[0,T]$ satisfying
(\ref{eq:---2}). 

Clearly, $(Y,Z)$ is a solution of (\ref{eq:-145}), and the proof
is completed.
\end{proof}

\subsection{\label{subsec:-6}BSDEs with stochastic durations}

In this subsection, we prove the existence and uniqueness of solutions
of the BSDE (\ref{eq:-146}). As in \citep{Pen91} (see also \citep[Section 7.3.2]{Yong99}),
we shall use the method of continuity. As before, we start with the
simple case where $g,f$ do not depend on $y$ or $z$, that is,
\begin{equation}
\left\{ \begin{aligned}dY_{t} & =-g(t)dt-f(t)d\langle W\rangle_{t}+Z_{t}dW_{t},\;\;t\in[0,\tau),\\
Y_{\tau} & =\xi,
\end{aligned}
\right.\label{eq:-168}
\end{equation}
where $g(t)=g(t,\omega)$ and $f(t)=f(t,\omega)$ are $\{\mathcal{F}_{t}^{\lambda}\}$-adapted
processes.
\begin{lem}
\label{lem:-22}Let $\beta=(\beta_{0},\beta_{1})\in[1,\infty)^{2}$,
and $\xi\in\mathcal{F}_{\tau}^{\lambda}$ which satisfy (\ref{eq:-196}).
Suppose that $g(t),f(t)$ are $\{\mathcal{F}_{t}^{\lambda}\}$-adapted
processes such that
\begin{equation}
\mathbb{E}_{\lambda}\Big(\int_{0}^{\tau}g(t)^{2}e^{2\beta_{0}t+2\beta_{1}\langle W\rangle_{t}}dt\Big)<\infty,\;\mathbb{E}_{\lambda}\Big(\int_{0}^{\tau}f(t)^{2}e^{2\beta_{0}t+2\beta_{1}\langle W\rangle_{t}}d\langle W\rangle_{t}\Big)<\infty.\label{eq:-150}
\end{equation}
Then BSDE (\ref{eq:-168}) admits a unique solution $(Y,Z)$ in $\mathcal{V}_{\lambda}^{\beta}[0,\tau]$.
Moreover,
\begin{equation}
\begin{aligned}\big\Vert(Y,Z)\big\Vert_{\mathcal{V}_{\lambda}^{\beta}[0,\tau]}^{2} & \le C\,\Big[\mathbb{E}_{\lambda}\big(\xi^{2}e^{2\beta_{0}\tau+2\beta_{1}\langle W\rangle_{\tau}}\big)+\mathbb{E}_{\lambda}\Big(\int_{0}^{\tau}g(t)^{2}e^{2\beta_{0}t+2\beta_{1}\langle W\rangle_{t}}dt\Big)\\
 & \quad+\mathbb{E}_{\lambda}\Big(\int_{0}^{\tau}f(t)^{2}e^{2\beta_{0}t+2\beta_{1}\langle W\rangle_{t}}d\langle W\rangle_{t}\Big)\Big],
\end{aligned}
\label{eq:-167}
\end{equation}
where $C>0$ is a constant depending only on $\beta$.
\end{lem}
\begin{proof}
Let $M_{t}=\mathbb{E}_{\lambda}\big(\xi+\int_{0}^{\tau}g(r)dr+\int_{0}^{\tau}f(r)d\langle W\rangle_{r}\,\big|\,\mathcal{F}_{t}^{\lambda}\big),\;t\ge0$.
Then, by (\ref{eq:-196}) and (\ref{eq:-150}), $M_{t}$ is an $\{\mathcal{F}_{t}^{\lambda}\}$-adapted
square-integrable martingale and $\mathbb{E}_{\lambda}(M_{\tau}^{2})<\infty$.
By Theorem \ref{thm:-2}, there exists a unique $\{\mathcal{F}_{t}^{\lambda}\}$-predictable
process $Z$ such that $M_{t}-M_{0}=\int_{0}^{t}Z_{r}dW_{r},\;t\ge0$. 

Let $Y_{t}=M_{t\wedge\tau}-\int_{0}^{t\wedge\tau}g(r)\,dr-\int_{0}^{t\wedge\tau}f(r)\,dW_{r},\;t\ge0$.
Then (\ref{eq:-147}) is satisfied. Let $\mathrm{e}_{t}=\exp(2\beta_{0}t+2\beta_{1}\langle W\rangle_{t}),\;t\ge0$.
By the definitions of $Y_{t}$ and $M_{t}$, we have
\[
|Y_{T\wedge\tau}-\xi|^{2}\mathrm{e}_{T\wedge\tau}=\Big|\mathbb{E}_{\lambda}\Big(\xi+\int_{T\wedge\tau}^{\tau}g(r)dr+\int_{T\wedge\tau}^{\tau}f(r)d\langle W\rangle_{r}\,\Big|\,\tilde{\mathcal{F}}_{T\wedge\tau}^{\lambda}\Big)-\xi\Big|^{2}\mathrm{e}_{T\wedge\tau},
\]
which implies that
\[
\begin{aligned}\mathbb{E}_{\lambda}\big(|Y_{T\wedge\tau}-\xi|^{2}\mathrm{e}_{T\wedge\tau}\big) & \le3\mathbb{E}_{\lambda}\big[\big|\mathbb{E}_{\lambda}\big(\xi|\tilde{\mathcal{F}}_{T\wedge\tau}^{\lambda}\big)-\xi\big|^{2}\mathrm{e}_{T\wedge\tau}\big]\\
 & \quad+3\mathbb{E}_{\lambda}\Big(\int_{T\wedge\tau}^{\tau}g(r)^{2}\mathrm{e}_{r}dr\Big)+3\mathbb{E}_{\lambda}\Big(\int_{T\wedge\tau}^{\tau}f(r)^{2}\mathrm{e}_{r}d\langle W\rangle_{r}\Big).
\end{aligned}
\]
By the Lebesgue dominated convergence theorem, the last two expectations
on the right hand side of the above converge to zero as $T\to\infty$.
For the first expectation on the right hand side of the above, note
that $\lim_{T\to\infty}\mathbb{E}_{\lambda}\big(\xi|\mathcal{F}_{T\wedge\tau}^{\lambda}\big)=\xi\;\;\mathbb{P}_{\lambda}$-a.s.
by the martingale convergence theorem, which, together with (\ref{eq:-196})
and the dominated convergence theorem, implies that
\[
\lim_{T\to\infty}\mathbb{E}_{\lambda}\big[\big|\mathbb{E}_{\lambda}\big(\xi|\mathcal{F}_{T\wedge\tau}^{\lambda}\big)-\xi\big|^{2}\mathrm{e}_{T\wedge\tau}\big]=0.
\]
This completes the proof of (\ref{eq:-151}).

The proof of (\ref{eq:-167}) is similar to that of (\ref{eq:-21}).
As a corollary of (\ref{eq:-167}), we see that $(Y,Z)$ is the unique
solution of (\ref{eq:-168}) in $\mathcal{V}_{\lambda}^{\beta}[0,\tau]$.
\end{proof}
Next, we consider the following BSDE parametrized by $\alpha\in[0,1]$:
\begin{equation}
\left\{ \begin{aligned}dY_{t} & =-\big(g_{0}(t)+\alpha g(t,Y_{t})\big)dt\\
 & \quad-\big(f_{0}(t)+\alpha f(t,Y_{t},Z_{t})\big)d\langle W\rangle_{t}+Z_{t}dW_{t},\;\;t\in[0,\tau),\\
Y_{\tau} & =\xi,
\end{aligned}
\right.\label{eq:-169}
\end{equation}
where $g_{0}(t)$ and $f_{0}(t)$ are $\{\mathcal{F}_{t}^{\lambda}\}$-adapted
processes, and $f,g$ are functions satisfying the measurability condition
$\hyperlink{M}{\text{(M)}}$ in Section \ref{sec:-6}. The following
a priori estimate is crucial to the proof of existence and uniqueness
of solutions of BSDEs with random durations.
\begin{lem}
\label{prop:-16}Let $\xi,\bar{\xi}\in\mathcal{F}_{\tau}^{\lambda}$
satisfy (\ref{eq:-196}). Let $g_{0},f_{0}$ and $\bar{g}_{0},\bar{f}_{0}$
be $\{\mathcal{F}_{t}^{\lambda}\}$-adapted processes satisfying (\ref{eq:-150}).
Suppose that $g,f$ satisfy (\ref{eq:-173})\textendash (\ref{eq:-193})
with $\kappa_{0}=0,\;\kappa_{1}=\frac{K_{1}^{2}}{2}$. Let $(Y,Z)$
be a solution in $\mathcal{V}_{\lambda}^{\beta}[0,\tau]$ of (\ref{eq:-169}),
and $(\bar{Y},\bar{Z})$ be a solution in $\mathcal{V}_{\lambda}^{\beta}[0,\tau]$
of the BSDE obtained by replacing $(\xi,g_{0},f_{0})$ by $(\bar{\xi},\bar{g}_{0},\bar{f}_{0})$
in (\ref{eq:-169}). Then
\begin{equation}
\Vert(\hat{Y},\hat{Z})\Vert_{\mathcal{V}_{\lambda}^{\beta}[0,\tau]}^{2}\le C\mathbb{E}_{\lambda}\Big(\int_{0}^{\tau}\hat{g}_{0}(r)^{2}e^{2\beta_{0}r+2\beta_{1}\langle W\rangle_{r}}dr+\int_{0}^{\tau}\hat{f}_{0}(r)^{2}e^{2\beta_{0}r+2\beta_{1}\langle W\rangle_{r}}d\langle W\rangle_{r}\Big),\label{eq:-182}
\end{equation}
where $\hat{\eta}=\eta-\bar{\eta}$ for $\eta=g_{0},f_{0},Y,Z$, and
$C>0$ is a constant depending only on $\beta$.
\end{lem}
\begin{proof}
Let $\mathrm{e}_{t}=\exp(2\beta_{0}t+2\beta_{1}\langle W\rangle_{t})$,
and $\hat{g}(t)=g(t,Y_{t})-g(t,\bar{Y}_{t}),\;\hat{f}(t)=f(t,Y_{t},Z_{t})-f(t,\bar{Y}_{t},\bar{Z}_{t})$.
By Itô's formula and by using (\ref{eq:-173})\textendash (\ref{eq:-172}),
we have
\[
\begin{aligned}\hat{Y}_{t\wedge\tau}^{2}\mathrm{e}_{t\wedge\tau} & \le\hat{Y}_{T\wedge\tau}^{2}\mathrm{e}_{T\wedge\tau}+\int_{t\wedge\tau}^{T\wedge\tau}\big(2\hat{Y}_{r}\hat{g}_{0}(r)-2\beta_{0}\hat{Y}_{r}^{2}\big)\mathrm{e}_{r}dr\\
 & \quad+\int_{t\wedge\tau}^{T\wedge\tau}\big[2\hat{Y}_{r}\hat{f}_{0}(r)+2\alpha K_{1}|\hat{Y}_{r}|\,|\hat{Z}_{r}|-(2\alpha\kappa_{1}+2\beta_{1})\hat{Y}_{r}^{2}-\hat{Z}_{r}^{2}\big]\mathrm{e}_{r}\,d\langle W\rangle_{r}\\
 & \quad-2\int_{t\wedge\tau}^{T\wedge\tau}\hat{Y}_{r}\hat{Z}_{r}\mathrm{e}_{r}dW_{r}\\
 & \le\hat{Y}_{T\wedge\tau}^{2}\mathrm{e}_{T\wedge\tau}-\beta_{0}\int_{t\wedge\tau}^{T\wedge\tau}\hat{Y}_{r}^{2}\mathrm{e}_{r}dr+\frac{1}{\beta_{0}}\int_{t\wedge\tau}^{T\wedge\tau}\hat{g}_{0}(r)^{2}\mathrm{e}_{r}dr\\
 & \quad+\int_{t\wedge\tau}^{T\wedge\tau}\Big[\big(a-2\alpha\kappa_{1}+b\alpha^{2}K_{1}^{2}-2\beta_{1}\big)\hat{Y}_{r}^{2}+\Big(\frac{1}{b}-1\Big)\hat{Z}_{r}^{2}+\frac{1}{a}\hat{f}_{0}(r)^{2}\Big]\mathrm{e}_{r}d\langle W\rangle_{r}\\
 & \quad-2\int_{t\wedge\tau}^{T\wedge\tau}\hat{Y}_{r}\hat{Z}_{r}\mathrm{e}_{r}dW_{r},
\end{aligned}
\]
where $a$ and $b$ are positive constants to be determined. 

Since $\kappa_{1}=\frac{K_{1}^{2}}{2}$, we may choose $b>1$ sufficiently
close to $1$, and choose accordingly $a>0$ sufficiently small such
that $a-2\kappa_{1}+bK_{1}^{2}-2\beta_{1}<0$. Since $\alpha\mapsto a-2\alpha\kappa_{1}+b\alpha^{2}K_{1}^{2}-2\beta_{1}$
is convex and is negative at $\alpha=0$ and $\alpha=1$, we see that
$a-2\alpha\kappa_{1}+b\alpha^{2}K_{1}^{2}-2\beta_{1}<0$ for each
$\alpha\in[0,1]$. With such $a$ and $b$, (\ref{eq:-182}) follows
easily from an argument similar to the proof of (\ref{eq:-21}).
\end{proof}
\begin{cor}
\label{cor:-7}Let $g,f$ satisfy (\ref{eq:-173})\textendash (\ref{eq:-193}).
Then there exists an $\epsilon_{0}>0$, depending only on $K_{0},K_{1}$
and $\beta$, such that the following holds: If, for some $\alpha\in[0,1]$,
(\ref{eq:-169}) admits a unique solution $(Y,Z)$ in $\mathcal{V}_{\lambda}^{\beta}[0,\tau]$
such that
\begin{equation}
\begin{aligned}\Vert(Y,Z)\Vert_{\mathcal{V}_{\lambda}^{\beta}[0,\tau]}^{2} & \le C\mathbb{E}_{\lambda}\Big(\xi^{2}e^{2\beta_{0}\tau+2\beta_{1}\langle W\rangle_{\tau}}+\int_{0}^{\tau}\big(g_{0}(t)^{2}+g(t,0)^{2}\big)e^{2\beta_{0}t+2\beta_{1}\langle W\rangle_{t}}dt\\
 & \quad+\int_{0}^{\tau}\big(f_{0}(t)^{2}+f(t,0,0)^{2}\big)e^{2\beta_{0}t+2\beta_{1}\langle W\rangle_{t}}d\langle W\rangle_{t}\Big),
\end{aligned}
\label{eq:-185}
\end{equation}
for any $\xi$ satisfying (\ref{eq:-196}) and any $g_{0},f_{0}$
satisfying (\ref{eq:-150}), where $C>0$ is a constant depending
only on $K_{0},K_{1}$ and $\beta$, then the same is valid when replacing
$\alpha$ by $\alpha+\epsilon$ with $\epsilon\in[0,\epsilon_{0}]$
and $\alpha+\epsilon\le1$, and (\ref{eq:-185}) holds for a possibly
different constant $C>0$ depending only on $K_{0},K_{1}$ and $\beta$.
\end{cor}
\begin{proof}
Suppose that (\ref{eq:-169}) admits a unique solution in $\mathcal{V}_{\lambda}^{\beta}[0,\tau]$
satisfying (\ref{eq:-185}) for some $\alpha\in[0,1]$. Let $\epsilon>0$
and $(Y_{0},Z_{0})=(0,0)$. By (\ref{eq:-183}) and Lemma \ref{lem:-22},
define inductively $(Y^{(n)},Z^{(n)}),\;n\in\mathbb{N}_{+}$ as the
unique solution in $\mathcal{V}_{\lambda}^{\beta}[0,\tau]$ of the
following BSDE
\[
\left\{ \begin{aligned}dY_{t}^{(n)} & =-\big[g_{0}(t)+\epsilon g(t,Y_{t}^{(n-1)})+\alpha g(t,Y_{t}^{(n)})\big]dt\\
 & \quad-\big[f_{0}(t)+\epsilon f(t,Y_{t}^{(n-1)},Z_{t}^{(n-1)})+\alpha f(t,Y_{t}^{(n)},Z^{(n)})\big]d\langle W\rangle_{t}\\
 & \quad+Z_{t}^{(n)}dW_{t},\hspace{4em}t\in[0,\tau),\\
Y_{\tau}^{(n)} & =\xi.
\end{aligned}
\right.
\]
According to (\ref{eq:-185}) and Lemma \ref{prop:-16}, we have
\begin{equation}
\begin{aligned}\Vert(Y^{(1)},Z^{(1)})\Vert_{\mathcal{V}_{\lambda}^{\beta}[0,\tau]}^{2} & \le C\,\mathbb{E}_{\lambda}\Big(\xi^{2}e^{2\beta_{0}\tau+2\beta_{1}\langle W\rangle_{\tau}}+\int_{0}^{\tau}\big(g_{0}(t)^{2}+g(t,0)^{2}\big)e^{2\beta_{0}t+2\beta_{1}\langle W\rangle_{t}}dt\\
 & \quad+\int_{0}^{\tau}\big(f_{0}(t)^{2}+f(t,0,0)^{2}\big)e^{2\beta_{0}t+2\beta_{1}\langle W\rangle_{t}}d\langle W\rangle_{t}\Big),
\end{aligned}
\label{eq:-186}
\end{equation}
and
\[
\Vert(Y^{(n+1)}-Y^{(n)},Z^{(n+1)}-Z^{(n)})\Vert_{\mathcal{V}_{\lambda}^{\beta}[0,\tau]}\le\epsilon C\,\Vert(Y^{(n)}-Y^{(n-1)},Z^{(n)}-Z^{(n-1)})\Vert_{\mathcal{V}_{\lambda}^{\beta}[0,\tau]},\;\;n\in\mathbb{N}_{+},
\]
where $C>0$ is a constant depending only on $K_{0},K_{1}$ and $\beta$.
In particular, $C$ is independent of $\alpha$ or $\epsilon$. Let
$\epsilon_{0}=(4C)^{-1/2}$. Then for each $\epsilon\in[0,\epsilon_{0}]$
with $\alpha+\epsilon\le1$,
\[
\Vert(Y^{(n+1)}-Y^{(n)},Z^{(n+1)}-Z^{(n)})\Vert_{\mathcal{V}_{\lambda}^{\beta}[0,\tau]}\le2^{-n}\,\Vert(Y^{(1)},Z^{(1)})\Vert_{\mathcal{V}_{\lambda}^{\beta}[0,\tau]},\quad n\in\mathbb{N}_{+}.
\]
This implies that $\lim_{n\to\infty}\Vert(Y^{(n)}-Y,Z^{(n)}-Z)\Vert_{\mathcal{V}_{\lambda}^{\beta}[0,\tau]}=0$
for some $(Y,Z)\in\mathcal{V}_{\lambda}^{\beta}[0,\tau]$.

Clearly, $(Y,Z)$ is the unique solution in $\mathcal{V}_{\lambda}^{\beta}[0,\tau]$
for the BSDE obtained by replacing $\alpha$ by $\alpha+\epsilon$
in (\ref{eq:-169}). Moreover, $\Vert(Y,Z)\Vert_{\mathcal{V}_{\lambda}^{\beta}[0,\tau]}\le2\Vert(Y^{(1)},Z^{(1)})\Vert_{\mathcal{V}_{\lambda}^{\beta}[0,\tau]}$.
This, together with (\ref{eq:-186}), completes the proof.
\end{proof}
We now give the proof of Theorem \ref{thm:-2-1}.

\begin{proof}[Proof of Theorem \ref{thm:-2-1}]
(a) This can be proved similarly to Theorem \ref{thm:-1-1}-(a).

(b) Suppose first that $\kappa_{0}=0,\;\kappa_{1}=K_{2}^{2}/2$. By
Lemma \ref{lem:-22}, when $\alpha=0$,
\[
\left\{ \begin{aligned}dY_{t} & =-\alpha g(t,X_{t},Y_{t})dt-\alpha f(t,X_{t},Y_{t},Z_{t})d\langle W\rangle_{t}+Z_{t}dW_{t},\;\;t\in[0,\tau),\\
Y_{\tau} & =\xi,
\end{aligned}
\right.
\]
admits a unique solution $(Y,Z)\in\mathcal{V}_{\lambda}^{\beta}[0,\tau]$
satisfying $\Vert(Y,Z)\Vert_{\mathcal{V}_{\lambda}^{\beta}[0,\tau]}^{2}\le C\mathbb{E}_{\lambda}\big(\xi^{2}e^{2\beta_{0}\tau+2\beta_{1}\langle W\rangle_{\tau}}\big)$,
where and thereafter, $C>0$ denotes any instance of a generic constant
depending only on $\kappa_{0},\kappa_{1},K_{0},K_{1},\beta$.

By Corollary \ref{cor:-7}, there exists an $\epsilon_{0}>0$ depending
only on $K_{0},K_{1},\beta$ and satisfying the property stated in
Corollary \ref{cor:-7}. Applying Corollary \ref{cor:-7} successively
shows that \ref{eq:-146} admits a unique solution $(Y,Z)\in\mathcal{V}_{\lambda}^{\beta}[0,\tau]$
satisfying (\ref{eq:-194}).

For the general case, we use the exponential martingale $\tilde{\mathrm{e}}_{t}=\exp\big[-\kappa_{0}t+\big(\frac{K_{1}^{2}}{2}-\kappa_{1}\big)\langle W\rangle_{t}\big],\;t\ge0$,
and let
\[
\tilde{\xi}=\xi\tilde{\mathrm{e}}_{\tau},\;\tilde{g}(t,y)=g(t,y\tilde{\mathrm{e}}_{t})\tilde{\mathrm{e}}_{t}^{-1},\;\tilde{f}(t,y,z)=f(t,y\tilde{\mathrm{e}}_{t},z\tilde{\mathrm{e}}_{t})\tilde{\mathrm{e}}_{t}^{-1},
\]
for each $t\ge0,\;y,z\in\mathbb{R}$. Then $\tilde{g},\tilde{f}$
satisfy the assumptions of the above case. Let
\[
\tilde{\beta}_{0}=\beta_{0}-\kappa_{0}>0,\;\tilde{\beta}_{1}=\beta_{1}-\kappa_{1}+\frac{K_{1}^{2}}{2}>0.
\]
Then
\[
\left\{ \begin{aligned}d\tilde{Y}_{t} & =-\tilde{g}(t,\tilde{Y}_{t})dt-\tilde{f}(t,\tilde{Y}_{t},\tilde{Z}_{t})d\langle W\rangle_{t}+\tilde{Z}_{t}dW_{t},\;\;t\in[0,\tau),\\
\tilde{Y}_{\tau} & =\tilde{\xi},
\end{aligned}
\right.
\]
admits a unique solution $(\tilde{Y},\tilde{Z})\in\mathcal{V}_{\lambda}^{\tilde{\beta}}[0,\tau]$
such that
\begin{align*}
\begin{aligned}\Vert(\tilde{Y},\tilde{Z})\Vert_{\mathcal{V}_{\lambda}^{\tilde{\beta}}[0,\tau]}^{2} & \le C\,\mathbb{E}_{\lambda}\Big(\tilde{\xi}^{2}e^{2\tilde{\beta}_{0}\tau+2\tilde{\beta}_{1}\langle W\rangle_{\tau}}+\int_{0}^{\tau}\tilde{g}(t,0)^{2}e^{2\tilde{\beta}_{0}t+2\tilde{\beta}_{1}\langle W\rangle_{t}}dt\\
 & \quad+\int_{0}^{\tau}\tilde{f}(t,0,0)^{2}e^{2\tilde{\beta}_{0}t+2\tilde{\beta}_{1}\langle W\rangle_{t}}d\langle W\rangle_{t}\Big)\\
 & =C\,\mathbb{E}_{\lambda}\Big(\xi^{2}e^{2\beta_{0}\tau+2\beta_{1}\langle W\rangle_{\tau}}+\int_{0}^{\tau}g(t,0)^{2}e^{2\beta_{0}t+2\beta_{1}\langle W\rangle_{t}}dt\\
 & \quad+\int_{0}^{\tau}f(t,0,0)^{2}e^{2\beta_{0}t+2\beta_{1}\langle W\rangle_{t}}d\langle W\rangle_{t}\Big).
\end{aligned}
\end{align*}
Let $Y_{t}=\tilde{Y}_{t}\tilde{\mathrm{e}}_{t}^{-1},\,Z_{t}=\tilde{Z}_{t}\tilde{\mathrm{e}}_{t}^{-1},\;t\ge0$.
It is easily seen that $(Y,Z)\in\mathcal{V}_{\lambda}^{\beta}[0,\tau]$
is a solution of (\ref{eq:-146}), and $(Y,Z)$ satisfies (\ref{eq:-194}).
Thus we have completed the proof.
\end{proof}

\subsection{\label{subsec:-1}Example}

Let $\lambda\in\mathcal{P}(\mathbb{S})$, and let $\tau$ be an $\{\mathcal{F}_{t}^{\lambda}\}$-stopping
time such that $\tau\le T\;\mathbb{P}_{\lambda}\text{-a.s.}$ for
some $T>0$. We present a worked-out solution of linear BSDEs on $(\Omega,\{\mathcal{F}_{t}^{\lambda}\},\mathbb{P}_{\lambda})$.
In contrast with BSDEs on Euclidean spaces, linear BSDEs on $\mathbb{S}$
can take the following two forms:
\[
\left\{ \begin{aligned}dY_{t} & =-aY_{t}dt-cZ_{t}d\langle W\rangle_{t}+Z_{t}dW_{t},\;\;t\in[0,\tau),\\
Y_{\tau} & =\xi,
\end{aligned}
\right.
\]
and
\[
\left\{ \begin{aligned}dY_{t} & =-(aY_{t}+cZ_{t})d\langle W\rangle_{t}+Z_{t}dW_{t},\;\;t\in[0,\tau),\\
Y_{\tau} & =\xi,
\end{aligned}
\right.
\]
where $a,c\in\mathbb{R}$ are constants, and $\xi\in L^{p}(\mathcal{F}_{\tau}^{\lambda},\mathbb{P}_{\lambda})$
for some $p>2$. Clearly, these cases can be unified as
\begin{equation}
\left\{ \begin{aligned}dY_{t} & =-aY_{t}dt-(bY_{t}+cZ_{t})d\langle W\rangle_{t}+Z_{t}dW_{t},\;\;t\in[0,\tau),\\
Y_{\tau} & =\xi,
\end{aligned}
\right.\label{eq:-7}
\end{equation}
where $a,b,c\in\mathbb{R}$ are constants.

To solve (\ref{eq:-7}), let
\[
\Phi_{t}=\exp\Big[at+\Big(b-\frac{c^{2}}{2}\Big)\langle W\rangle_{t}+cW_{t}\Big],\quad t\ge0.
\]
Then, by Lemma \ref{lem:-30} and Corollary \ref{cor:}.(a), $\mathbb{E}_{\lambda}(\Phi_{t}^{q})<\infty$
for any $t\ge0$ and any $q>0$. 

By Itô's formula,
\begin{equation}
d\Phi_{t}=a\Phi_{t}dt+b\Phi_{t}d\langle W\rangle_{t}+c\Phi_{t}dW_{t},\quad t\ge0.\label{eq:-8}
\end{equation}
Furthermore, if $(Y,Z)$ is the solution of (\ref{eq:-7}) then, by
(\ref{eq:-7}) and (\ref{eq:-8}),
\[
d(\Phi_{t}Y_{t})=Y_{t}d\Phi_{t}+\Phi_{t}dY_{t}+d\langle\Phi,Y\rangle_{t}=\Phi_{t}(cY_{t}+Z_{t})dW_{t}.
\]
Therefore,
\[
\Phi_{t\wedge\tau}Y_{t\wedge\tau}=\Phi_{\tau}\xi-\int_{t\wedge\tau}^{\tau}\Phi_{r}(cY_{r}+Z_{r})dW_{r},\quad t\ge0.
\]
Taking conditional expectations on both sides of the above gives that
\[
\Phi_{t}Y_{t}=\Phi_{t\wedge\tau}Y_{t\wedge\tau}=\mathbb{E}_{\lambda}(\Phi_{\tau}\xi\vert\mathcal{F}_{t\wedge\tau}^{\lambda})=\mathbb{E}_{\lambda}(\Phi_{\tau}\xi\vert\mathcal{F}_{t}^{\lambda}),\quad t\ge0.
\]
Equivalently,
\begin{equation}
Y_{t}=\Phi_{t}^{-1}\mathbb{E}_{\lambda}(\Phi_{\tau}\xi\vert\mathcal{F}_{t}^{\lambda}),\quad t\ge0.\label{eq:-9}
\end{equation}

Since $\xi\in L^{p}(\mathcal{F}_{\tau}^{\lambda},\mathbb{P}_{\lambda})$
for some $p>2$, $\Phi_{\tau}\xi\in L^{2}(\mathcal{F}_{\tau}^{\lambda},\mathbb{P}_{\lambda})$.
Therefore, by Theorem \ref{thm:-2}, there exists a unique $\{\mathcal{F}_{t}^{\lambda}\}$-predictable
process $\zeta(t)$ such that
\begin{equation}
\mathbb{E}_{\lambda}(\Phi_{\tau}\xi\vert\mathcal{F}_{t}^{\lambda})=\mathbb{E}_{\lambda}(\Phi_{\tau}\xi)+\int_{0}^{t}\zeta(r)dW_{r},\quad t\ge0.\label{eq:-17}
\end{equation}
By a similar argument for (\ref{eq:-8}), we have
\[
d\Phi_{t}^{-1}=-a\Phi_{t}^{-1}dt-(b-c^{2})\Phi_{t}^{-1}d\langle W\rangle_{t}-c\Phi_{t}^{-1}dW_{t},\quad t\ge0.
\]
By (\ref{eq:-9}), (\ref{eq:-17}) and the above equation, 
\begin{equation}
\begin{aligned}dY_{t} & =\Phi_{t}^{-1}\zeta(t)dW_{t}+\mathbb{E}_{\lambda}(\Phi_{\tau}\xi\vert\mathcal{F}_{t}^{\lambda})d\Phi_{t}^{-1}-c\Phi_{t}^{-1}\zeta(t)d\langle W\rangle_{t}\\
 & =-aY_{t}dt-[bY_{t}+c(\Phi_{t}^{-1}\zeta(t)-cY_{t})]d\langle W\rangle_{t}+(\Phi_{t}^{-1}\zeta(t)-cY_{t})dW_{t}.
\end{aligned}
\label{eq:-34}
\end{equation}
Let
\begin{equation}
Z_{t}=\Phi_{t}^{-1}\zeta(t)-cY_{t}=\Phi_{t}^{-1}[\zeta(t)-c\mathbb{E}_{\lambda}(\Phi_{\tau}\xi\vert\mathcal{F}_{t}^{\lambda})],\quad t\ge0.\label{eq:-28}
\end{equation}
Then, by (\ref{eq:-34}), $(Y,Z)$ given by (\ref{eq:-9}) and (\ref{eq:-28})
is the solution of (\ref{eq:-7}).

\section{\label{sec:-4}Representations for solutions of semi-linear parabolic
PDEs}

In this section, we give the proof of Theorem \ref{thm:-3-1}.
\begin{proof}[Proof of Theorem \ref{thm:-3-1}]
We prove the theorem by several steps.\\
\\
\textbf{Step 1.}\emph{\enskip{}Let}
\[
g^{(s)}(r,x)=g(r+s,x,u(r+s,x)),\;f^{(s)}(r,x)=f(r+s,x,u(r+s,x),\nabla u(t+s,x)).
\]
\emph{Then, for any $\eta\in\mathcal{F}(\mathbb{S}\backslash\mathrm{V}_{0})$,
\begin{equation}
\begin{aligned}\frac{d}{dt}\mathbb{E}_{\mu}\big[u\big(t\wedge\sigma^{(s)}+s,X_{t\wedge\sigma^{(s)}}\big)\eta(X_{0})\big] & =\big\langle\partial_{t}u(t+s),P_{t}^{0}\eta\big\rangle_{\mu}\\
 & \quad-\mathcal{E}(u(t+s),P_{t}^{0}\eta)\;\;\mathit{a.e.}\;t\in[0,T],
\end{aligned}
\label{eq:-133}
\end{equation}
\begin{equation}
\frac{d}{dt}\mathbb{E}_{\mu}\Big[\Big(\int_{0}^{t\wedge\sigma^{(s)}}g^{(s)}(r,X_{r})\,dr\Big)\eta(X_{0})\Big]=\big\langle g^{(s)}(t),P_{t}^{0}\eta\big\rangle_{\mu}\;\;\mathit{a.e.}\ t\in[0,T],\label{eq:-139}
\end{equation}
and
\begin{equation}
\frac{d}{dt}\mathbb{E}_{\mu}\Big[\Big(\int_{0}^{t\wedge\sigma^{(s)}}f^{(s)}(r,X_{r})d\langle W\rangle_{r}\Big)\eta(X_{0})\Big]=\langle f^{(s)}(t),P_{t}^{0}\eta\rangle_{\nu}\;\;\mathit{a.e.}\;t\in[0,T].\label{eq:-140}
\end{equation}
}\\
\emph{Proof of Step 1}.\emph{\enskip{}}Without loss of generality,
we may assume $s=0$. Let $H(\varphi(t))$ be the harmonic function
with boundary value $\varphi(t)$ and $u^{0}(t,x)=u(t,x)-H(\varphi(t))(x)$.
Let $0<\delta<T-t$. Since
\[
u^{0}(t\wedge\sigma_{\mathrm{V}_{0}},X_{t\wedge\sigma_{\mathrm{V}_{0}}})=u^{0}(t,X_{t})1_{\{t<\sigma_{\mathrm{V}_{0}}\}},
\]
and $u^{0}(t)=0$ on $\mathrm{V}_{0}$, using the $\mu$-symmetry
of $\{X_{t}^{0}\}$, we obtain
\begin{equation}
\begin{aligned}\mathbb{E}_{\mu}\big[\big(u^{0}((t+\delta)\wedge\sigma_{\mathrm{V}_{0}},X_{(t+\delta)\wedge\sigma_{\mathrm{V}_{0}}})\eta(X_{0})\big] & =\mathbb{E}_{\mu}\big[u^{0}(t+\delta,X_{0}^{0})\eta(X_{t+\delta}^{0})\big]\\
 & =\big\langle u^{0}(t+\delta),P_{t+\delta}^{0}\eta\big\rangle_{\mu}.
\end{aligned}
\label{eq:-10}
\end{equation}
Similarly, $\mathbb{E}_{\mu}\big[\big(u^{0}(t\wedge\sigma_{\mathrm{V}_{0}},X_{t\wedge\sigma_{\mathrm{V}_{0}}})\eta(X_{0})\big]=\big\langle u^{0}(t),P_{t}^{0}\eta\big\rangle_{\mu}$.
Therefore,
\begin{equation}
\begin{aligned}\mathbb{E}_{\mu}\big[\big(u((t+\delta)\wedge\sigma_{\mathrm{V}_{0}},X_{(t+\delta)\wedge\sigma_{\mathrm{V}_{0}}}) & -u(t\wedge\sigma_{\mathrm{V}_{0}},X_{t\wedge\sigma_{\mathrm{V}_{0}}})\big)\eta(X_{0})\big]=\big\langle u^{0}(t+\delta)-u^{0}(t),P_{t+\delta}^{0}\eta\big\rangle_{\mu}\\
 & +\big\langle u^{0}(t),P_{t+\delta}^{0}\eta-P_{t}^{0}\eta\big\rangle_{\mu}+\langle H[\varphi(t+\delta)-\varphi(t)],P_{t}^{0}\eta\rangle_{\mu}.
\end{aligned}
\label{eq:-11}
\end{equation}

Note that $u^{0}(t)\in\mathcal{F}(\mathbb{S}\backslash\mathrm{V}_{0})$.
We have
\begin{equation}
\lim_{\delta\to0}\frac{1}{\delta}\big\langle u^{0}(t),P_{t+\delta}^{0}\eta-P_{t}^{0}\eta\big\rangle_{\mu}=-\mathcal{E}(u^{0}(t),P_{t}^{0}\eta)=-\mathcal{E}(u(t),P_{t}^{0}\eta),\label{eq:-12}
\end{equation}
where we have used in the second equality the fact that $\mathcal{E}(H(\varphi(t)),v)=0$
for any $v\in\mathcal{F}(\mathbb{S}\backslash\mathrm{V}_{0})$.

By Lemma \ref{lem:-1} and $\lim_{\delta\to0}\Vert P_{t+\delta}^{0}\eta-P_{t}^{0}\eta\Vert_{L^{2}(\mu)}=0$,
we deduce that
\begin{equation}
\begin{aligned}\lim_{\delta\to0}\frac{1}{\delta}\big\langle u^{0}(t+\delta)-u^{0}(t),P_{t+\delta}^{0}\eta\big\rangle_{\mu} & =\lim_{\delta\to0}\frac{1}{\delta}\big\langle u(t+\delta)-u(t),P_{t+\delta}^{0}\eta\big\rangle_{\mu}-\langle H(\partial_{t}\varphi(t)),P_{t}^{0}\eta\rangle_{\mu}\\
 & =\langle\partial_{t}u(t),P_{t}^{0}\eta\rangle_{\mu}-\langle H(\partial_{t}\varphi(t)),P_{t}^{0}\eta\rangle_{\mu}.
\end{aligned}
\label{eq:-25}
\end{equation}

The equality (\ref{eq:-133}) now follows readily from (\ref{eq:-11}),
(\ref{eq:-12}) and (\ref{eq:-25}).\bigskip{}

Similar to (\ref{eq:-10}), we have
\[
\mathbb{E}_{\mu}\Big[\Big(\int_{t\wedge\sigma_{\mathrm{V}_{0}}}^{(t+\delta)\wedge\sigma_{\mathrm{V}_{0}}}g^{(0)}(r,X_{r})\,dr\Big)\eta(X_{0})\Big]=\int_{0}^{\delta}\big\langle g^{(0)}(t+r),P_{t+r}^{0}\eta\big\rangle_{\mu}dr.
\]
Using the Lebesgue dominated convergence theorem and fact that $\lim_{r\to0}P_{t+r}^{0}\eta(x)=P_{t}^{0}\eta(x),\,x\in\mathbb{S}\backslash\mathrm{V}_{0}$,
we obtain (\ref{eq:-139}).\bigskip{}

We now prove (\ref{eq:-140}). Similar to the above, by the $\mu$-symmetry
of $\{X_{t}^{0}\}$ again, we have
\[
\mathbb{E}_{\mu}\Big[\Big(\int_{t\wedge\sigma_{\mathrm{V}_{0}}}^{(t+\delta)\wedge\sigma_{\mathrm{V}_{0}}}f^{(0)}(r,X_{r})\,d\langle W\rangle_{r}\Big)\eta(X_{0})\Big]=\mathbb{E}_{\mu}\Big[\Big(\int_{0}^{\delta}f^{(0)}(t_{n}+r,X_{r}^{0})d\langle W\rangle_{r}\Big)P_{t}^{0}\eta(X_{0}^{0})\Big].
\]
Now we apply Lemma \ref{lem:-28} and conclude that
\[
\frac{d}{dt}\mathbb{E}_{\mu}\Big[\Big(\int_{0}^{t}f^{(0)}(r,X_{r}^{0})\,d\langle W\rangle_{r}\Big)\eta(X_{0})\Big]=\lim_{\delta\to0}\frac{1}{\delta}\int_{0}^{\delta}\big\langle f^{(0)}(r+t),P_{t}^{0}\eta\big\rangle_{\nu}dr=\big\langle f^{(0)}(t),P_{t}^{0}\eta\big\rangle_{\nu}.
\]
This completes the proof of Step 1.\\
\\
\textbf{Step 2.}\emph{\enskip{}Let
\[
\begin{aligned}M_{t}^{(s)} & =u(t\wedge\sigma^{(s)}+s,X_{t\wedge\sigma^{(s)}})-u(s,X_{0})+\int_{0}^{t\wedge\sigma^{(s)}}g^{(s)}(r,X_{r})dr\\
 & \quad+\int_{0}^{t\wedge\sigma^{(s)}}f^{(s)}(r,X_{r})d\langle W\rangle_{r},\;\;t\ge0.
\end{aligned}
\]
Then $\{M_{t}^{(s)}\}$ is a $\mathbb{P}_{x}$-martingale for each
$x\in\mathbb{S}\backslash\mathrm{V}_{0}$.}\\
\emph{}\\
\emph{Proof of Step 2}.\emph{\enskip{}}By Step 1,
\[
\frac{d}{dt}\mathbb{E}_{\mu}\big[M_{t}^{(s)}\eta(X_{t_{0}}^{0})\big]=0\;\;\text{for all}\ \eta\in\mathcal{F}(\mathbb{S}\backslash\mathrm{V}_{0})\;\;\text{a.e.}\ t\ge t_{0},
\]
which, together with the continuity of $t\mapsto\mathbb{E}_{x}(M_{t}^{(s)})$,
implies that
\begin{equation}
\mathbb{E}_{x}\big(M_{t}^{(s)}\big)=0\;\;\text{for all}\ t\ge0\;\;\mu\text{-a.e.}\;x\in\mathbb{S}\backslash\mathrm{V}_{0}.\label{eq:-198}
\end{equation}

We claim that $x\mapsto\mathbb{E}_{x}(M_{t}^{(s)}),\;x\in\mathbb{S}\backslash\mathrm{V}_{0}$
is continuous, and therefore, (\ref{eq:-198}) holds for all $t\ge0$
and all $x\in\mathbb{S}\backslash\mathrm{V}_{0}$. Note that
\[
M_{t}^{(s)}=u(t+s,X_{t}^{0})-u(s,X_{0}^{0})+\int_{0}^{t}g^{(s)}(r,X_{r}^{0})\,dr+\int_{0}^{t}f^{(s)}(r,X_{r}^{0})\,d\langle W\rangle_{r},\;\;0\le t\le T-s.
\]
Recall that, in the above, we have used the convention that $\eta(\Delta)=0$
for any function $\eta$ on $\mathbb{S}$. Hence
\[
\begin{aligned}\mathbb{E}_{x}\big(M_{t}^{(s)}\big) & =P_{t}^{0}\big(u^{0}(t+s)\big)(x)-u^{0}(s,x)+\int_{0}^{t}\big[P_{r}^{0}\big(g^{(s)}(r)\big)(x)+P_{r}^{0}\big(f^{(s)}(r)\nu\big)(x)\big]dr.\end{aligned}
\]
Therefore, it suffices to prove the continuity of $x\mapsto\int_{0}^{t}P_{r}^{0}\big(f^{(s)}(r)\nu\big)(x)dr,\;x\in\mathbb{S}\backslash\mathrm{V}_{0}$.

By Lemma \ref{lem:-29}-(b) and (\ref{eq:-189}), we have
\begin{equation}
\begin{aligned}\Vert P_{r}^{0}\big(f^{(s)}(r)\nu\big)\Vert_{L^{\infty}} & \le C\,\min\{1,r^{-d_{s}/2}\}\Vert f^{(s)}(r)\Vert_{L^{1}(\nu)}\\
 & \le C\,\min\{1,r^{-d_{s}/2}\}\big[\Vert f(r+s,0,0)\Vert_{L^{1}(\nu)}+\Vert u(r+s)\Vert_{L^{1}(\nu)}\\
 & \quad+\Vert\nabla u(r+s)\Vert_{L^{1}(\nu)}\big]\\
 & \le C\,\min\{1,r^{-d_{s}/2}\}\big[1+\max_{[0,T]\times\mathbb{S}}|u|+\mathcal{E}(u(r+s))^{1/2}\big]
\end{aligned}
\label{eq:-31}
\end{equation}
for all $r>0$, where $C>0$ is a constant depending only on $K_{0},K_{1}$
and $\max_{[0,T]\times\mathbb{S}}|f(t,x,0,0)|$. Note that $\int_{0}^{T}\mathcal{E}(u(t))\,dt<\infty$,
and that, for each $r>0$,
\[
P_{r}^{0}\big(f^{(s)}(r)\nu\big)(x)=\int_{\mathbb{S}}f^{(s)}(r,y)p_{r}^{0}(x,y)\nu(dy)
\]
is continuous in $x\in\mathbb{S}\backslash\mathrm{V}_{0}$. Now the
continuity of $x\mapsto\int_{0}^{t}P_{r}^{0}\big(f^{(s)}(r)\nu\big)(x)dr$
follows readily from (\ref{eq:-31}) and the Lebesgue dominated convergence
theorem.

Therefore, $\mathbb{E}_{x}\big(M_{t}^{(s)}\big)=0$ for all $t\ge0,\,x\in\mathbb{S}\backslash\mathrm{V}_{0}$,
which, together with the Markov property of $\{X_{t}^{0}\}$, competes
the proof of Step 2.\\
\\
\textbf{Step 3.}\emph{\enskip{}For each $s\in[0,T)$ and each $x\in\mathbb{S}\backslash\mathrm{V}_{0}$,
\[
(Y_{t}^{(s)},Z_{t}^{(s)})=(u(t\wedge\sigma^{(s)}+s,X_{t\wedge\sigma^{(s)}}),\nabla u(t\wedge\sigma^{(s)}+s,X_{t\wedge\sigma^{(s)}}))
\]
is the solution of the BSDE (\ref{eq:-3}). Moreover, the representation
(\ref{eq:-22}) holds, and the solution of (\ref{eq:}) is unique.}\\
\emph{}\\
\emph{Proof of Step 3}.\emph{\enskip{}}By Lemma \ref{lem:},
\[
u(t+s,X_{t})=u(s,X_{0})+\int_{0}^{t}\nabla u(r+s,X_{r})dW_{r}+N_{t}^{(s)},\quad t\ge0,
\]
where $N^{(s)}$ is a continuous process with zero quadratic variation.
Let
\[
Q_{t}^{(s)}=N_{t\wedge\sigma^{(s)}}^{(s)}+\int_{0}^{t\wedge\sigma^{(s)}}g^{(s)}(r,X_{r})dr+\int_{0}^{t\wedge\sigma^{(s)}}f^{(s)}(r,X_{r})d\langle W\rangle_{r},\;\;t\ge0.
\]
Then
\begin{equation}
Q_{t}^{(s)}=M_{t}^{(s)}-\int_{0}^{t\wedge\sigma^{(s)}}\nabla u(r+s,X_{r})dW_{r},\;\;t\ge0,\label{eq:-197}
\end{equation}
and therefore $\{Q_{t}^{(s)}\}$ is a $\mathbb{P}_{x}$-martingale
for all $x\in\mathbb{S}\backslash\mathrm{V}_{0}$.

Let $t_{i}=\frac{i}{n}t,\;0\le i\le n$. Then
\[
\lim_{n\to\infty}\sum_{i=1}^{n}(Q_{t_{i}}^{(s)}-Q_{t_{i-1}}^{(s)})^{2}=\langle Q^{(s)}\rangle_{t}\;\;\mbox{in}\;L^{1}(\mathbb{P}_{x})\;\;\mbox{for each}\;x\in\mathbb{S}\backslash\mathrm{V}_{0}.
\]
Since $\lim_{n\to\infty}\mathbb{E}_{\mu}\big[\sum_{i=1}^{n}(N_{t_{i}}^{(s)}-N_{t_{i-1}}^{(s)})^{2}\big]=0$,
there exists a subsequence $\{n_{k}\}$ such that 
\[
\lim_{k\to\infty}\sum_{i=1}^{n_{k}}(N_{t_{i}}^{(s)}-N_{t_{i-1}}^{(s)})^{2}=0\;\;\mathbb{P}_{\mu}\mbox{-a.s.},
\]
which implies that $\langle Q^{(s)}\rangle_{t}=0\;\;\mathbb{P}_{\mu}$-a.s.,
as $t\mapsto\int_{0}^{t\wedge\sigma^{(s)}}g^{(s)}(r,X_{r})dr$ and
$t\mapsto\int_{0}^{t\wedge\sigma^{(s)}}f^{(s)}(r,X_{r})d\langle W\rangle_{r}$
are of bounded variation.

Therefore, $\mathbb{E}_{x}\big(\langle Q\rangle_{t}^{(s)}\big)=0\;\;\mu\mbox{-a.e.}\;x\in\mathbb{S}\backslash\mathrm{V}_{0}$.
By an argument similar to the proof of Step 2, it can be shown that
$x\mapsto\mathbb{E}_{x}\big(\langle Q\rangle_{t}^{(s)}\big)$ is continuous.
Therefore, $\mathbb{E}_{x}\big(\langle Q\rangle_{t}^{(s)}\big)=0$
for all $x\in\mathbb{S}\backslash\mathrm{V}_{0}$. In particular,
by (\ref{eq:-197}),
\[
\begin{aligned}u(t\wedge\sigma^{(s)}+s,X_{t\wedge\sigma^{(s)}}) & =u(s,X_{0})-\int_{0}^{t\wedge\sigma^{(s)}}g^{(s)}(r,X_{r})dr-\int_{0}^{t\wedge\sigma^{(s)}}f^{(s)}(r,X_{r})d\langle W\rangle_{r}\\
 & \quad+\int_{0}^{t\wedge\sigma^{(s)}}\nabla u(r+s,X_{r})dW_{r}\quad\mathbb{P}_{x}\text{-a.s.}\ \ \text{for all}\ x\in\mathbb{S}\backslash\mathrm{V}_{0}.
\end{aligned}
\]
This implies that 
\[
(Y_{t}^{(s)},Z_{t}^{(s)})=(u(t\wedge\sigma^{(s)}+s,X_{t\wedge\sigma^{(s)}}),\nabla u(t\wedge\sigma^{(s)}+s,X_{t\wedge\sigma^{(s)}}))
\]
is the unique solution of the BSDE (\ref{eq:-3}) on $\big(\Omega,\{\mathcal{F}_{t}\},\mathbb{P}_{x}\big)$
for $x\in\mathbb{S}\backslash\mathrm{V}_{0}$, which is clearly also
valid for $x\in\mathrm{V}_{0}$. As a result, we obtain the representation
(\ref{eq:-22}).

The uniqueness of the solution of (\ref{eq:}) follows immediately
from the representation (\ref{eq:-22}) and the uniqueness of solutions
of (\ref{eq:-3}).
\end{proof}
\begin{example*}
As an application of the representation formula given in Theorem \ref{thm:-3-1},
we solve the following parabolic equation
\[
\left\{ \begin{aligned}(\partial_{t}u & +\mathcal{L}u)\cdot\mu=-au\cdot\mu-(bu+c\nabla u)\cdot\nu,\;\;(t,x)\in[0,T)\times\mathbb{S}\backslash\mathrm{V}_{0},\\
u(t & ,x)=\varphi(t,x)\;\text{on}\;[0,T)\times\mathrm{V}_{0},\;\;u(T)=\psi,
\end{aligned}
\right.
\]
where $a,b,c\in\mathbb{R}$ are constants.

Let 
\[
\Psi(t,x)=\begin{cases}
\varphi(t,x), & \mbox{if }(t,x)\in[0,T)\times\mathrm{V}_{0},\\
\psi(x), & \mbox{if }(t,x)\in\{T\}\times\mathbb{S}\backslash\mathrm{V}_{0}.
\end{cases}
\]
By the example in Section \ref{subsec:-1}, we see that the solution
of the BSDE
\[
\left\{ \begin{aligned}dY_{t}^{(s)} & =-aY_{t}^{(s)}dt-(bY_{t}^{(s)}+cZ_{t}^{(s)})d\langle W\rangle_{t}+Z_{t}^{(s)}dW_{t},\;\;t\in[0,\sigma^{(s)}),\\
Y_{\sigma^{(s)}}^{(s)} & =\Psi(\sigma^{(s)},X_{\sigma^{(s)}}),
\end{aligned}
\right.
\]
on $\big(\Omega,\mathcal{F},\{\mathcal{F}_{t}\}_{t\ge0},\mathbb{P}_{x}\big)$
is given by
\[
Y_{t}^{(s)}=\Phi_{t\wedge\sigma^{(s)}}^{-1}\mathbb{E}_{x}\big[\Phi_{\sigma^{(s)}}\Psi\big(\sigma^{(s)},X_{\sigma^{(s)}}\big)\vert\mathcal{F}_{t}\big],\;\;t\in[0,T-s],
\]
where
\[
\Phi_{t}=\exp\Big[at+\Big(b-\frac{c^{2}}{2}\Big)\langle W\rangle_{t}+cW_{t}\Big],\quad t\ge0.
\]
Therefore
\[
u(t,x)=\mathbb{E}_{x}(Y_{0}^{(t)})=\mathbb{E}_{x}\big[\Phi_{(T-t)\wedge\sigma_{\mathrm{V}_{0}}}\Psi\big((T-t)\wedge\sigma_{\mathrm{V}_{0}},X_{(T-t)\wedge\sigma_{\mathrm{V}_{0}}}\big)\big],\;\;t\in[0,T],\;x\in\mathbb{S}.
\]
In particular, if $\varphi=0$, then
\[
u(t,x)=\mathbb{E}_{x}\big[\Phi_{T-t}\psi\big(X_{T-t}^{0}\big)\big],\;\;t\in[0,T],\;x\in\mathbb{S}.
\]
\end{example*}
\begin{rem}
\label{rem:-5}It is well known that, for $\mathbb{R}^{d}$, solutions
of BSDEs correspond to viscosity solutions of corresponding PDEs,
which is a very weak formulation of solutions. Moreover, Theorem \ref{thm:-3-1}
shows that solutions of BSDEs correspond to the solution of the PDE
(\ref{eq:}) whenever a solution exists. These justify to name the
functions given by (\ref{eq:-22}) the \emph{viscosity solutions}
of (\ref{eq:}). The existence of such very weak solutions is guaranteed
by Theorem \ref{thm:-1-1}. On the other hand, we also note that the
existence of solutions of BSDEs does not imply the existence of weak
solutions of the corresponding semi-linear parabolic equations. In
future work we shall explore the existence of weak solutions of (\ref{eq:})
on the Sierpinski gasket.
\end{rem}

\section*{Acknowledgements\addcontentsline{toc}{section}{Acknowledgements}}

The authors want to thank Professor Hino for his helpful suggestions
on the first version of the paper. We would also like to express our
gratitude to the anonymous referee for the invaluable comments and
suggestions (in particular, providing many related papers), which
help improve the quality of the presentation.

\renewcommand\bibname{References} 
\end{document}